\documentclass[10pt]{article}

\usepackage{amsmath,amssymb,amsthm,mathrsfs,dsfont}
\usepackage{cite}
\usepackage[margin=3cm]{geometry} %¿í°æÊ±ÓÃ£¬ÅäºÏË«±¶ÐÐ¾à

\usepackage{titlesec,hyperref}
\usepackage{tocloft}
\usepackage{color}
\usepackage{fancyhdr}
\titleformat{\subsection}{\bfseries}{\thesubsection.\enspace}{1pt}{}

\newtheorem{theo}{Theorem}[section]
\newtheorem{lemm}[theo]{Lemma}
\newtheorem{defi}[theo]{Definition}
\newtheorem{coro}[theo]{Corollary}

\numberwithin{equation}{section}

\allowdisplaybreaks %ÔÊÐí¹«Ê½·ÖÒ³ÏÔÊ¾

\begin{document}
	\title{Existence and uniqueness of the global conservative solutions for the generalized Camassa-Holm equation with dual-power nonlinearities
		\hspace{-4mm}
	}
	
	\author{Jian $\mbox{Chen}^1$\footnote{E-mail: chenj883@mail2.sysu.edu.cn},\quad
		Xiaoxin $\mbox{Chen}^1$\footnote{E-mail: chenxx233@mail2.sysu.edu.cn},\quad
		Zhaoyang $\mbox{Yin}^{1}$\footnote{E-mail: mcsyzy@mail.sysu.edu.cn}\\
		$^1\mbox{School}$ of Science,\\ Shenzhen Campus of Sun Yat-sen University, Shenzhen 518107, China}
	\date{}
	\maketitle
	\hrule
	
	\begin{abstract}
		In this paper, we investigate the global conservative solutions to the generalized Camassa-Holm equation with dual-power nonlinearities. By introducing a new set of variables, we transform the original equation into an equivalent semi-linear system, which allows us to establish the global existence of conservative solutions. Furthermore, for a given global conservative solution, we construct some auxiliary variables tailored to its specific structure and demonstrate that they satisfy a semi-linear system with a unique solution, thereby deriving the uniqueness of conservative solutions to the original equation.
		
		\vspace*{5pt}
		
		\vspace*{5pt}
	\end{abstract}
	
	%\vspace*{3pt}
	
	%\phantomsection
	%\addcontentsline{toc}{section}{\contentsname}
	%添加目录到书签
	\tableofcontents
	
	\section{Introduction}
	
	In this paper, we are mainly concerned with the global conservative solutions for the generalized Camassa-Holm equation with dual-power nonlinearities
	\begin{equation}\label{eq0}
		u_t-\mu u_{txx}+2ku_x+\eta u_{xxx}=f(u)u_x+s(2u_xu_{xx}+uu_{xxx}),
	\end{equation}
	where $f(u)=Au+Bu^m,\ \mu>0,\ k,\eta,s,A,B\in\mathbb{R},\ m\in\mathbb{N}$ and the function $u=u(t,x)$ stands for the average fluid velocity. \eqref{eq0} was first proposed by Nanta et al. in \cite{nanta2021identification} as a new water wave model with important physical significance, in which they also derived a class of analytical solitary wave solutions and rigorously proved second-order convergence of the numerical scheme. Recently, Dong \cite{dong2024wave} studied the local well-posedness, blow-up mechanism and wave-breaking phenomena for the Cauchy problem associated to \eqref{eq0}. Qiu et al. \cite{qiu2022traveling}  utilized the dynamical system method to classify the traveling wave solutions of \eqref{eq0}, including peakons, compactons, and solitary waves. Li et al. \cite{li2024classification} further provided a complete classification of all traveling wave solutions.
	
	When $\eta=\mu=1,\ A=-1,\ k=s=B=0$, \eqref{eq0} reduces to the BBM-KdV equation
	\begin{align}\label{BBM-Kdv}
		u_t+u_{txx}+u_{xxx}+uu_x=0,
	\end{align}
	which was originally proposed by Bona and Smith \cite{bona1975initial}, serving as a fundamental model for unidirectional waves incorporating dissipation and nonlinear dispersion. Dutykh and Pelinovsky \cite{dutykh2014numerical} utilized numerical simulations to explore solitonic gas dynamics and collective soliton behaviors in \eqref{BBM-Kdv}. Moreover, Mancas and Adams \cite{mancas2017elliptic} established both local and global well-posedness for the initial value problem of \eqref{BBM-Kdv} in $H^s(\mathbb{R})$ ($s \ge 1$) by contraction mapping and energy estimates.
	
	When $s=\mu,\ A+B=-3,\ m=1$, \eqref{eq0} reduces to the Dullin-Gottwald-Holm (DGH) equation
	\begin{align}\label{DGH}
		u_t-\mu u_{txx}+2ku_x+\mu u_{xxx}+3uu_x-\mu (2u_xu_{xx}+uu_{xxx})=0,
	\end{align}
	which was derived by Dullin et al.\cite{dullin2001integrable} from water wave theory via asymptotic analysis and near-identity normal form transformations, describing the unidirectional propagation of surface waves in the shallow water regime. \eqref{DGH} possesses a bi-Hamiltonian structure and admits exact peakons \cite{dullin2001integrable}. The local well-posedness of the Cauchy problem for \eqref{DGH} in $H^s(\mathbb{R})$ ($s > 3/2$), along with the limiting behavior of solutions as $\mu \to 0$, the stability of solitary waves, and scattering data, were investigated in \cite{tian2005well}. Furthermore, blow-up mechanism and wave-breaking phenomena for both line and periodic cases to \eqref{DGH} have been presented respectively in \cite{liu2006global,zhang2008blow}.
	
	When $\eta=B=0,\ A=-3,\ \mu=s=1$, \eqref{eq0} reduces to the celebrated Camassa-Holm (CH) equation
	\begin{align}\label{CH}
		u_t-u_{txx}+3uu_x=2u_xu_{xx}+uu_{xxx},
	\end{align}
	which was derived in \cite{camassa1993integrable,fuchssteiner1981symplectic}, models the unidirectional propagation of shallow water waves under gravity, and is renowned for its peaked solitons. In the last three decades, the CH equation has been studied extensively from various perspectives. The literature covers a wide range of topics, ranging from complete integrability \cite{fisher1999camassa, fuchssteiner1981symplectic} and geometric structures \cite{kouranbaeva1999camassa, misiolek1998shallow,MR2216268} to well-posedness issues \cite{danchin2001few,danchin2003note,li2000well}. Furthermore, the wave-breaking phenomena \cite{constantin2000existence,constantin1998global,constantin1998wave,constantin2000blow,MR4851884}, stability of solutions \cite{constantin2000stability,constantin2002stability}, and ill-posedness \cite{GuoYe2022ill,GuoLiu2019ill,li2025ill} have been thoroughly studied.
	
	%and the existence of global conservative \cite{Constantin2007conservative} and dissipative solutions \cite{bressan2007global} have been thoroughly studied.
	
	To extend solutions beyond finite-time wave breaking, Bressan and Constantin \cite{Constantin2007conservative} pioneered a characteristic method to establish the global existence of conservative solutions for the CH equation. Building upon this, Bressan et al. \cite{bressan2014uniqueness} subsequently proved the uniqueness of such solutions by developing a direct inverse route that singles out a unique characteristic curve. These seminal frameworks have been successfully adapted to other models, including, but not limited to, the Constantin-Lannes equation \cite{MR2481064,MR3147227,MR4643656}, the rotation Camassa-Holm equation\cite{MR3936045,MR3912736} and the nonlinear dispersive wave equations \cite{MR2292515,MR4697983,MR4888793}.
	
	Motivated by the works mentioned above, we herein investigate the global existence and uniqueness of conservative solutions to \eqref{eq0}. It is convenient to rewrite the Cauchy problem of \eqref{eq0} as
	\begin{equation}\label{eq1}
		\left\{
		\begin{array}{ll}
			u_t+\left(\frac{s}{\mu}u-\frac{\eta}{\mu}\right)u_x
			=-P_x,\ &t>0,\ x\in\mathbb{R}, \\
			u(0,x)=u_0(x), &x\in\mathbb{R},
		\end{array}\right.
	\end{equation}
	where the nonlocal source term $P$ is defined by
	\begin{equation}
		P\triangleq \frac{1}{2\sqrt{\mu}}e^{-\frac{|\cdot|}{\sqrt{\mu}}}*\left(-\left(\frac{A}{2}+\frac{s}{2\mu}\right)u^2+\frac{s}{2}u_x^2-\frac{B}{m+1}u^{m+1}+\left(2k+\frac{\eta}{\mu}\right)u\right).
	\end{equation}
	The existence proof relies on introducing new variables that resolve all finite-time singularities, yielding an equivalent semi-linear system. The major difficulty in our work is that \eqref{eq1} contains nonlocal higher-order nonlinear terms and the linear convolution contribution, which require more delicate analysis and estimates to establish the Lipschitz continuity as well as the extension of solutions for the associated semi-linear system. To overcome this difficulty, we establish refined estimates for the nonlocal terms in the new coordinates, in particular a uniform $L^\infty$-bound for the linear convolution term, which is achieved by adapting the variable transformation and employing a contradiction argument. These estimates yield the global solution of the associated semi-linear system, and reverting to the original variables then recovers the global conservative solutions. For uniqueness, we associate any global conservative solution with specific auxiliary variables. This allows us to derive an integral equation that determines a unique characteristic, ultimately reducing the uniqueness of $u$ to that of a uniquely solvable semi-linear system.
	
	The remainder of this paper is organized as follows. In Section \ref{Basic definitions}, we provide some preliminary definitions and state our main result. In Section \ref{Existence}, by introducing a new set of variables, we convert the original equation into an equivalent semi-linear system, from which the existence of global conservative solutions is derived. In Section \ref{Uniqueness}, we establish the uniqueness of solutions to \eqref{eq1} by constructing an auxiliary system of ordinary differential equations.
	
	\section{\textbf{Basic definitions and results}}\label{Basic definitions}
	In this section, we recall some definitions of global conservative solutions for \eqref{eq1} and give our main result.
	
	For smooth solutions, differentiating \eqref{eq1} with respect to $x$, we have
	\begin{equation}\label{ux}
		u_{xt}+\left(\frac{s}{\mu}u-\frac{\eta}{\mu}\right)u_{xx}
		=\frac{1}{\mu}\left(-\frac{s}{2}u_x^2-P-\left(\frac{A}{2}+\frac{s}{2\mu}\right)u^2-\frac{B}{m+1}u^{m+1}+\left(2k+\frac{\eta}{\mu}\right)u\right).
	\end{equation}
	It follows from \eqref{eq1} and \eqref{ux} that
	\begin{equation*}
		\left(u^2\right)_t+\left(\frac{2s}{3\mu}u^3-\frac{\eta}{\mu}u^2+2uP\right)_x=2u_xP,
	\end{equation*}
	\begin{equation}\label{ux2}
		\left(u_x^2\right)_t+\frac{1}{\mu}\left(\left(su-\eta\right)u_x^2+\left(\frac A3+\frac{s}{3\mu}\right)u^3+\frac{2B}{(m+1)(m+2)}u^{m+2}-\left(2k+\frac{\eta}{\mu}\right)u^2\right)_x=-\frac{2}{\mu}u_xP.
	\end{equation}
	Hence, for smooth solutions,
	\begin{equation}\label{conservative_E(t)}
		\int_{\mathbb{R}}u^2(t,x)+\mu u_x^2(t,x)dx
	\end{equation}
	is constant in time. Let $w=u_x^2$, then \eqref{ux2} implies
	\begin{equation}\label{eqw}
		w_t+\left(\left(\frac{s}{\mu}u-\frac{\eta}{\mu}\right)w\right)_x=\frac{2}{\mu}u_x\left(-\left(\frac A2+\frac{s}{2\mu}\right)u^2-\frac{B}{m+1}u^{m+1}+\left(2k+\frac{\eta}{\mu}\right)u-P\right).
	\end{equation}
	
	Before providing our main result, we give some definitions.
	\begin{defi}\label{solution}
		We call $u=u(t,x)$ a solution of the Cauchy problem \eqref{eq1} on $[0,T]$ if $u$ is a H\"{o}lder continuous function defined on $[0,T]\times\mathbb{R}$ with the following properties. \\
		1. $u(t,\cdot)\in H^1(\mathbb{R}),\ \forall\ t\in[0,T]$.\\
		2. the map $t\mapsto u(t,\cdot)$ is Lipschitz continuous from $[0,T]$ into $L^2(\mathbb{R})$ and satisfies the initial condition $u(0,x)=u_0(x)$ together with
		\begin{equation}\label{u}
			\frac{d}{dt}u=-\left(\frac{s}{\mu}u-\frac{\eta}{\mu}\right)u_x-P_x
		\end{equation}
		for a.e. $t$. Here \eqref{u} is understood as an equality between functions in $L^2(\mathbb{R})$.
	\end{defi}
	
	\begin{defi}\label{conservative}
		We call $u=u(t,x)$ a global conservative solution of \eqref{eq1} if $u$ satisfies the following properties. \\
		1. $u$ provides a solution to the Cauchy problem \eqref{eq1} on $[0,\infty)$ in sense of Definition \ref{solution}.\\
		2. There exists a family of bounded Radon measures $\left\{\mu_{(t)},\ t\geq0\right\}$, depending continuously on time with respect to the topology of weak convergence of measures, whose absolutely continuous part has density $u_x^2$ with respect to the Lebesgue measure. This family provides a measure-valued solution $w$ to the balance law \eqref{eqw}, namely
		\begin{align}
			&\int_{0}^{\infty}\int_\mathbb{R}\varphi_t+\varphi_x\left(\frac{s}{\mu}u-\frac{\eta}{\mu}\right)d\mu_{(t)}dt+\int_{\mathbb{R}}u_{0x}^{2}\varphi(0,x)dx\notag\\
			+&\frac{2}{\mu}\int_{0}^{\infty}\int_\mathbb{R}\left(-\left(\frac{A}{2}+\frac{s}{2\mu}\right)u^2-\frac{B}{m+1}u^{m+1}+\left(2k+\frac{\eta}{\mu}\right)u-P\right)u_x\varphi dxdt=0\label{w}
		\end{align}
		for every test function $\varphi\in\mathcal{C}_c^1(\mathbb{R}^+\times\mathbb{R})$.
	\end{defi}
	
	The main theorem of this paper is as follows.
	\begin{theo}
		Let $s\neq0$, then for any initial data $u_0\in H^1(\mathbb{R})$, the Cauchy problem \eqref{eq1} has a unique global conservative solution in the sense of
		Definition \ref{conservative}.
	\end{theo}
	
	\section{Existence}\label{Existence}
	This section is devoted to proving the existence of global conservative solutions for the Cauchy problem \eqref{eq1}. We first introduce an equivalent semi-linear system and establish the existence of its global solutions. This result is then applied to obtain the global conservative solutions for the original equation.
	\subsection{An equivalent semi-linear system}\label{sec_equivalent}
	Given the initial data $u_0\in H^1(\mathbb{R})$ and a new variable $\xi\in\mathbb{R}$, we define the non-decreasing map $\xi\mapsto y_0(\xi)$ by the following equation:
	\begin{equation}\label{y0}
		\int_0^{y_0(\xi)}(1+u_{0x}^2)dx=\xi.
	\end{equation}
	Assuming that the solution $u$ to \eqref{eq1} is Lipschitz continuous for $t\in[0,T]$, we now derive a system equivalent to \eqref{eq1} in the independent variables $(t,\xi)$. We define the characteristics $t\mapsto y(t,\xi)$ as the solutions of
	\begin{equation}\label{y}
		\frac{\partial}{\partial t} y(t,\xi)=\frac{s}{\mu}u\left(t,y(t,\xi)\right)-\frac{\eta}{\mu}, \quad y(0,\xi)=y_0(\xi).
	\end{equation}
	Introduce some new variables
	\begin{equation*}
		u(t,\xi)\triangleq u\left(t,y(t,\xi)\right),\quad P(t,\xi)\triangleq P\left(t,y(t,\xi)\right),\quad P_x(t,\xi)\triangleq P_x\left(t,y(t,\xi)\right),
	\end{equation*}
	\begin{equation}\label{vq}
		v(t,\xi)\triangleq 2\arctan u_x,\quad q(t,\xi)\triangleq (1+u_x^2)\cdot\frac{\partial y}{\partial\xi},
	\end{equation}
	with $u_{x}=u_{x}(t,y(t,\xi))$. From \eqref{y0}, we have
	$$q(0,\xi)\equiv1.$$
	A straightforward calculation establishes the following equalities:
	\begin{equation}\label{equalities}
		\frac{1}{1+u_{x}^{2}}=\cos^{2}\frac{v}{2},\quad\frac{u_{x}}{1+u_{x}^{2}}=\frac{1}{2}\sin v,\quad\frac{u_{x}^{2}}{1+u_{x}^{2}}=\sin^{2}\frac{v}{2},\quad\frac{\partial y}{\partial\xi}=\frac{q}{1+u_{x}^{2}}=\cos^{2}\frac{v}{2}\cdot q.
	\end{equation}
	From \eqref{equalities}, we obtain an expression for $P$ and $P_x$ in terms of the new variable $\xi$, namely
	\begin{align}
		P(\xi) =&\frac{1}{2\sqrt{\mu}}\int_{-\infty}^{\infty}\exp\left\{-\frac{1}{\sqrt{\mu}}\left|\int_{\xi}^{\xi^{\prime}}\cos^{2}\frac{v(s)}{2}\cdot q(s)ds\right|\right\}\notag \\
		& \cdot\left\{\left[-(\frac{A}{2}+\frac{s}{2\mu})u^2(\xi^{\prime})-\frac{B}{m+1}u^{m+1}(\xi^{\prime})+(2k+\frac{\eta}{\mu})u(\xi^{\prime})\right]\cos^2\frac{v(\xi^{\prime})}{2}+\frac{s}{2}\sin^2\frac{v(\xi^{\prime})}{2}\right\}q(\xi^{\prime})d\xi^{\prime},\label{P}\\
		P_x(\xi) =&\frac{1}{2\mu}\left(\int_\xi^\infty-\int_{-\infty}^\xi\right)\exp\left\{-\frac{1}{\sqrt{\mu}}\left|\int_{\xi}^{\xi^{\prime}}\cos^{2}\frac{v(s)}{2}\cdot q(s)ds\right|\right\}\notag \\
		& \cdot\left\{\left[-(\frac{A}{2}+\frac{s}{2\mu})u^2(\xi^{\prime})-\frac{B}{m+1}u^{m+1}(\xi^{\prime})+(2k+\frac{\eta}{\mu})u(\xi^{\prime})\right]\cos^2\frac{v(\xi^{\prime})}{2}+\frac{s}{2}\sin^2\frac{v(\xi^{\prime})}{2}\right\}q(\xi^{\prime})d\xi^{\prime}.\label{Px}
	\end{align}
	Owing to \eqref{eq1} and \eqref{y}, we deduce that
	\begin{equation*}
		\frac{\partial}{\partial t}u(t,\xi)=-P_x(t,\xi),
	\end{equation*}
	with $P_x$ given by \eqref{Px}. Similarly, using \eqref{y}-\eqref{equalities} and \eqref{ux}, we can obtain
	\begin{equation*}
		\frac{\partial}{\partial t}v(t,\xi)=\frac{1}{\mu}(1+\cos v)\left(-\left(\frac A2+\frac{s}{2\mu}\right)u^2-\frac{B}{m+1}u^{m+1}+\left(2k+\frac{\eta}{\mu}\right)u-P\right)-\frac{s}{\mu}\sin^2\frac{v}{2}
	\end{equation*}
	and
	\begin{equation*}
		\frac{\partial}{\partial t}q(t,\xi)=\frac{1}{\mu}\left(\frac{s}{2}-\left(\frac A2+\frac{s}{2\mu}\right)u^2-\frac{B}{m+1}u^{m+1}+\left(2k+\frac{\eta}{\mu}\right)u-P\right)\sin v\cdot q,
	\end{equation*}
	with $P$ given by \eqref{P}.
	
	\subsection{Global solutions of the semi-linear system}
	In this subsection, we study the global solutions of the semi-linear system \eqref{semi-linear system}. Based on the analysis in Section \ref{sec_equivalent}, we rewrite the Cauchy problem for the variables $(u,v,q)$ in the form
	\begin{equation}\label{semi-linear system}
		\left\{
		\begin{array}{l}
			\frac{\partial u}{\partial t}=-P_x,\\
			\frac{\partial v}{\partial t}=\frac{1}{\mu}(1+\cos v)\left(-\left(\frac A2+\frac{s}{2\mu}\right)u^2-\frac{B}{m+1}u^{m+1}+\left(2k+\frac{\eta}{\mu}\right)u-P\right)-\frac{s}{\mu}\sin^2\frac{v}{2},\\
			\frac{\partial q}{\partial t}=\frac{1}{\mu}\left(\frac{s}{2}-\left(\frac A2+\frac{s}{2\mu}\right)u^2-\frac{B}{m+1}u^{m+1}+\left(2k+\frac{\eta}{\mu}\right)u-P\right)\sin v\cdot q,
		\end{array}\right.
	\end{equation}
	with
	\begin{equation}\label{semi-linear system_initial data}
		\left\{
		\begin{array}{l}
			u(0,\xi)=u_0\left(y_0(\xi)\right), \\
			v(0,\xi)=2\arctan u_{0x}\left(y_0(\xi)\right), \\
			q(0,\xi)=1.
		\end{array}\right.
	\end{equation}
	We consider \eqref{semi-linear system} as an ordinary differential equation in the Banach space
	\begin{equation*}
		X\triangleq H^1(\mathbb{R})\times\left[L^2(\mathbb{R})\cap L^\infty(\mathbb{R})\right]\times L^\infty(\mathbb{R}),
	\end{equation*}
	with norm
	\begin{equation*}
		\|(u,v,q)\|_X\triangleq\|u\|_{H^1}+\|v\|_{L^2}+\|v\|_{L^\infty}+\|q\|_{L^\infty}.
	\end{equation*}
	\begin{theo}
		Let $u_0\in H^1(\mathbb{R})$. Then the Cauchy problem \eqref{semi-linear system}-\eqref{semi-linear system_initial data} has a unique global solution.
	\end{theo}
	\begin{proof}
		We consider only the case $s\neq0$, since the case $s=0$ is similar and even simpler.
		
		\textbf{Step 1: Local existence}
		
		By the standard theory of ordinary differential equations in Banach spaces, it suffices to prove that the right hand side of \eqref{semi-linear system} is Lipschitz continuous on every bounded domain $\Omega\subset X$, which is defined as
		\begin{align*}
			\Omega=\left\{(u,v,q):\ \|u\|_{H^{1}}\leq\alpha,\ \|v\|_{L^{2}}\leq\beta,\ \|v\|_{L^{\infty}}\leq\frac{3\pi}{2},\ q(x)\in[q^{-},q^{+}]\ \mathrm{for\ a.e.}x\in \mathbb{R}\right\},
		\end{align*}
		for any constants $\alpha,\beta,q^-,q^+>0$. It's easy to check that the maps
		$$(1+\cos v)\left(-\left(\frac A2+\frac{s}{2\mu}\right)u^2-\frac{B}{m+1}u^{m+1}+\left(2k+\frac{\eta}{\mu}\right)u\right),\quad\sin^2\frac{v}{2}$$
		and 
		$$\left(\frac{s}{2}-\left(\frac A2+\frac{s}{2\mu}\right)u^2-\frac{B}{m+1}u^{m+1}+\left(2k+\frac{\eta}{\mu}\right)u\right)\sin v\cdot q$$
		are all Lipschitz continuous from $\Omega$ into $L^2(\mathbb{R})\cap L^\infty(\mathbb{R})$. Hence, the proof reduces to showing that 
		\begin{equation}\label{maps_P_Px}
			(u,v,q)\mapsto P,\quad(u,v,q)\mapsto P_x
		\end{equation}
		are Lipschitz continuous from $\Omega$ into $H^1(\mathbb{R})$. Since $|v|\leq\frac{3\pi}{2}$ implies 
		$$\sin^2\frac{v}{2}\leq\frac{v^2}{4}\leq\frac{9\pi^2}{8}\sin^2\frac{v}{2},$$
		we have
		\begin{align*}
			\mathrm{meas}\left\{\xi\in\mathbb{R}:\left|\frac{v(\xi)}{2}\right|\geq\frac{\pi}{4}\right\}
			\leq18\int_{\left\{\xi\in\mathbb{R}:18\sin^{2}\frac{v(\xi)}{2}\geq1\right\}}\sin^{2}\frac{v(\xi)}{2}d\xi\leq\frac{9}{2}\beta^{2}.
		\end{align*}
		Thus, for any $\xi_1<\xi_2$, we obtain
		\begin{align*}
			\int_{\xi_{1}}^{\xi_{2}}\cos^{2}\frac{v(\xi)}{2}\cdot q(\xi)d\xi \geq\int_{\left\{\xi\in[\xi_{1},\xi_{2}]:\mid\frac{v(\xi)}{2}\mid\leq\frac{\pi}{4}\right\}}\frac{q^{-}}{2}d\xi \geq\left(\frac{\xi_{2}-\xi_{1}}{2}-\frac{9}{4}\beta^{2}\right)q^{-}. 
		\end{align*}
		Introducing the exponentially decaying function
		\begin{equation*}
			\Gamma(\zeta)\triangleq\min\left\{1,\exp\left[\frac{q^-}{\sqrt{\mu}}\left(\frac{9}{4}\beta^2-\frac{|\zeta|}{2}\right)\right]\right\},
		\end{equation*}
		it's easy to show that $\|\Gamma\|_{L^1}=9\beta^2+\frac{4\sqrt{\mu}}{q^-}<\infty$. 
		
		To establish the Lipschitz continuity of the maps given in \eqref{maps_P_Px}, we first prove that $P,\ P_x\in H^1(\mathbb{R})$, namely $P,\ \partial_{\xi}P,\ P_x,\ \partial_{\xi}P_x\in L^2(\mathbb{R})$. For simplicity, we only focus on the estimates for $P_x$ and $\partial_{\xi}P_x$, as the arguments for $P$ and $\partial_{\xi}P$ are similar. Since \eqref{Px} yields 
		$$
		\mid P_x\mid\leq\frac{q^+}{2\mu}\Gamma*\left| \left(-(\frac{A}{2}+\frac{s}{2\mu})u^2-\frac{B}{m+1}u^{m+1}+(2k+\frac{\eta}{\mu})u\right)\cos^2\frac{v}{2}+\frac{s}{2}\sin^2\frac{v}{2}\right|,
		$$
		we obtain
		$$
		\|P_x\|_{L^2}\leq\frac{q^+}{2\mu}\|\Gamma\|_{L^1}C(\|u\|_{L^\infty}\|u\|_{L^2}+\|u\|_{L^\infty}^m\|u\|_{L^2}+\|u\|_{L^2}+\|v\|_{L^\infty}\|v\|_{L^2})<\infty.
		$$
		Differentiating \eqref{Px} with respect to $\xi$, we see
		\begin{align}
			\frac{\partial}{\partial\xi}P_{x}(\xi) =&-\frac{1}{\mu}\left\{\left[-(\frac{A}{2}+\frac{s}{2\mu})u^2(\xi)-\frac{B}{m+1}u^{m+1}(\xi)+(2k+\frac{\eta}{\mu})u(\xi)\right]\cos^2\frac{v(\xi)}{2}+\frac{s}{2}\sin^2\frac{v(\xi)}{2}\right\}q(\xi) \notag\\
			& +\frac{1}{2\mu\sqrt{\mu}}\left(\int_{\xi}^{\infty}-\int_{-\infty}^{\xi}\right)\exp\left\{-\frac{1}{\sqrt{\mu}}\left|\int_{\xi}^{\xi^{\prime}}\cos^{2}\frac{v(s)}{2}\cdot q(s)ds\right|\right\} \notag\\
			& \cdot\left\{\left[-(\frac{A}{2}+\frac{s}{2\mu})u^2(\xi^{\prime})-\frac{B}{m+1}u^{m+1}(\xi^{\prime})+(2k+\frac{\eta}{\mu})u(\xi^{\prime})\right]\cos^2\frac{v(\xi^{\prime})}{2}+\frac{s}{2}\sin^2\frac{v(\xi^{\prime})}{2}\right\}q(\xi^{\prime})\notag\\
			&
			\cdot\mathrm{sign}(\xi^{\prime}-\xi)\cos^2\frac{v(\xi)}{2}\cdot q(\xi)d\xi^{\prime},\label{partial_xi Px}
		\end{align}
		which implies
		$$
		\|\partial_{\xi} P_{x}\|_{L^2}\leq C\left(\frac{q^+}{\mu}+\frac{q^+\cdot q^+}{2\mu\sqrt{\mu}}\|\Gamma\|_{L^1}\right)(\|u\|_{L^\infty}\|u\|_{L^2}+\|u\|_{L^\infty}^m\|u\|_{L^2}+\|u\|_{L^2}+\|v\|_{L^\infty}\|v\|_{L^2})<\infty.
		$$
		
		Next, we show that
		$$\frac{\partial P}{\partial u},\quad\frac{\partial P}{\partial v},\quad\frac{\partial P}{\partial q},\quad\frac{\partial P_x}{\partial u},\quad\frac{\partial P_x}{\partial v},\quad\frac{\partial P_x}{\partial q}$$
		are uniformly bounded linear operators from the appropriate spaces into $H^1(\mathbb{R})$. By way of example, the detailed estimates for $\frac{\partial P_x}{\partial u}$ are provided. The other derivatives can be handled analogously.
		
		For $(m,v,q)\in\Omega$, the linear operator $\frac{\partial P_x}{\partial u}$ is defined by
		\begin{align*}
			\left[\frac{\partial P_{x}(u,v,q)}{\partial u}\cdot\widetilde{u}\right](\xi) =&\frac{1}{2\mu}\left(\int_{\xi}^{\infty}-\int_{-\infty}^{\xi}\right)\exp\left\{-\frac{1}{\sqrt{\mu}}\left|\int_{\xi}^{\xi^{\prime}}\cos^{2}\frac{v(s)}{2}\cdot q(s)ds\right|\right\} \\
			& \cdot\left[-(A+\frac{s}{\mu})u(\xi^{\prime})-Bu^{m}(\xi^{\prime})+(2k+\frac{\eta}{\mu})\right]\cos^2\frac{v(\xi^{\prime})}{2}\cdot q(\xi^{\prime})\widetilde{u}(\xi)d\xi^{\prime}.
		\end{align*}
		Thus, 
		\begin{align*}
			\left\|\frac{\partial P_{x}(u,v,q)}{\partial u}\cdot\widetilde{u}\right\|_{L^2} \leq\frac{C}{2\mu}\|\widetilde{u}\|_{L^2}\|\Gamma\|_{L^1} \left(\|u\|_{L^\infty}+\|u\|_{L^\infty}^m+1\right)q^+\leq C\|\widetilde{u}\|_{H^1}.
		\end{align*}
		Similarly, the linear operator $\frac{\partial (\partial_{\xi}P_x)}{\partial u}$ is defined by
		\begin{align*}
			\left[\frac{\partial (\partial_{\xi}P_x)(u,v,q)}{\partial u}\cdot\widetilde{u}\right](\xi) =&-\frac{1}{\mu}\left[-(A+\frac{s}{\mu})u(\xi)-Bu^{m}(\xi)+(2k+\frac{\eta}{\mu})\right]\cos^2\frac{v(\xi)}{2}q(\xi)\widetilde{u}(\xi) \\
			& +\frac{1}{2\mu\sqrt{\mu}}\left(\int_{\xi}^{\infty}-\int_{-\infty}^{\xi}\right)\exp\left\{-\frac{1}{\sqrt{\mu}}\left|\int_{\xi}^{\xi^{\prime}}\cos^{2}\frac{v(s)}{2}\cdot q(s)ds\right|\right\} \\
			& \cdot\left[-(A+\frac{s}{\mu})u(\xi^{\prime})-Bu^{m}(\xi^{\prime})+(2k+\frac{\eta}{\mu})\right]\cos^2\frac{v(\xi^{\prime})}{2}q(\xi^{\prime})\\
			&
			\cdot\mathrm{sign}(\xi^{\prime}-\xi)\cos^2\frac{v(\xi)}{2}\cdot q(\xi)\widetilde{u}(\xi)d\xi^{\prime},
		\end{align*}
		which implies 
		\begin{align*}
			\left\|\frac{\partial (\partial_{\xi}P_x)(u,v,q)}{\partial u}\cdot\widetilde{u}\right\|_{L^2} \leq C\left(\frac{1}{\mu}+\frac{q^+}{2\mu\sqrt{\mu}}\|\Gamma\|_{L^1}\right) \left(\|u\|_{L^\infty}+\|u\|_{L^\infty}^m+1\right)q^+\|\widetilde{u}\|_{L^2}\leq C\|\widetilde{u}\|_{H^1}.
		\end{align*}
		Therefore, $\frac{\partial P_x}{\partial u}$ is a uniformly linear operator from $H^1(\mathbb{R})$ into $H^1(\mathbb{R})$. 
		
		In conclusion, the semi-linear system \eqref{semi-linear system}-\eqref{semi-linear system_initial data} admits a unique solution in $C([0,T];X)$ for some $T>0$.
		
		\textbf{Step 2: Extension to a global solution}
		
		The global extension of the local solution to \eqref{semi-linear system}-\eqref{semi-linear system_initial data} is guaranteed provided
		\begin{equation}\label{extension_5}
			\|q(t)\|_{L^\infty}+\|\frac{1}{q(t)}\|_{L^\infty}+\|v(t)\|_{L^2}+\|v(t)\|_{L^\infty}+\|u(t)\|_{H^1}
		\end{equation}
		stays uniformly bounded on every bounded time interval. The proof is based on the conservation law \eqref{conservative_E(t)} and a contradiction argument. 
		
		Firstly, we claim that
		\begin{equation}\label{u_xi}
			u_\xi=\frac{q}{2}\sin v.
		\end{equation}
		Indeed, 
		\begin{align*}
			\left(\frac{q}{2}\sin v\right)_t=&\frac{q_t}{2}\sin v+\frac{q}{2}\cos v\cdot v_t\\
			=&\frac{q}{2\mu}\sin^2 v\left(\frac{s}{2}-\left(\frac A2+\frac{s}{2\mu}\right)u^2-\frac{B}{m+1}u^{m+1}+\left(2k+\frac{\eta}{\mu}\right)u-P\right)\\
			&+\frac{q}{2\mu}\cos v\left(2\cos^2\frac{v}{2}\left(-\left(\frac A2+\frac{s}{2\mu}\right)u^2-\frac{B}{m+1}u^{m+1}+\left(2k+\frac{\eta}{\mu}\right)u-P\right)-s\sin^2\frac{v}{2}\right)\\
			=&\frac{q}{\mu}\cos^2\frac{v}{2}\left(-\left(\frac A2+\frac{s}{2\mu}\right)u^2-\frac{B}{m+1}u^{m+1}+\left(2k+\frac{\eta}{\mu}\right)u-P\right)+\frac{sq}{2\mu}\sin^2\frac{v}{2}.
		\end{align*}
		Combining this result with \eqref{P}, \eqref{semi-linear system} and \eqref{partial_xi Px}, we obtain
		$$\left(\frac{q}{2}\sin v\right)_t=u_{\xi t},$$
		which, together with $\left(\frac{q}{2}\sin v-u_{\xi}\right)\mid_{t=0}=0$, yields \eqref{u_xi}.
		
		Next, we prove that
		\begin{equation}\label{conservative_E(t)_new variables}
			E(t)=\int_{\mathbb{R}}\left(u^2\cos^2\frac{v}{2}+\mu\sin^2\frac{v}{2}\right)qd\xi=E(0)\triangleq E_0.
		\end{equation}
		In fact, from \eqref{P}-\eqref{Px}, we infer
		$$P_\xi=q\cdot\cos^2\frac{v}{2}\cdot P_x.$$
		Thus, using \eqref{semi-linear system} and \eqref{u_xi}, we have
		\begin{align*}
			&\frac{d}{dt}\int_{\mathbb{R}}\left(u^2\cos^2\frac{v}{2}+\mu\sin^2\frac{v}{2}\right)qd\xi\\
			=&\int_{\mathbb{R}}\left(u^2\cos^2\frac{v}{2}+\mu\sin^2\frac{v}{2}\right)q_td\xi+\int_{\mathbb{R}}2uu_t\cos^2\frac{v}{2}\cdot qd\xi+\frac{1}{2}\int_{\mathbb{R}}(\mu-u^2)\sin v\cdot v_tqd\xi\\
			=&\int_{\mathbb{R}}\sin v\cdot q\left(-\frac{A}{2}u^2-\frac{B}{m+1}u^{m+1}+\left(2k+\frac{\eta}{\mu}\right)u-P\right)-2u\cos^2\frac{v}{2}\cdot qP_xd\xi\\
			=&-\int_{\mathbb{R}}\partial_{\xi}\left(2uP+\frac{A}{3}u^3+\frac{2B}{(m+1)(m+2)}u^{m+2}-\left(2k+\frac{\eta}{\mu}\right)u^2\right)d\xi=0,
		\end{align*}
		which proves \eqref{conservative_E(t)_new variables}.
		
		Therefore, as long as the solution is defined, \eqref{u_xi} and \eqref{conservative_E(t)_new variables} give
		\begin{equation}\label{bdd_u}
			\sup_{\xi\in\mathbb{R}}|u^{2}(t,\xi)| \leq 2\int_\mathbb{R}|uu_\xi|d\xi  \leq2\int_\mathbb{R}|u||\sin\frac{v}{2}\cos\frac{v}{2}|qd\xi\leq \frac{1}{\sqrt{\mu}}E_0.
		\end{equation}
		It follows from \eqref{P}, \eqref{conservative_E(t)_new variables} and \eqref{bdd_u} that
		\begin{align*}
			\|P\|_{L^\infty}\leq&\frac{1}{2\sqrt{\mu}}\left(2|k|+\frac{|\eta|}{\mu}\right)\left\|\int_{-\infty}^{\infty}\exp\left\{-\frac{1}{\sqrt{\mu}}\left|\int_{\xi}^{\xi^{\prime}}\cos^{2}\frac{v(s)}{2}\cdot q(s)ds\right|\right\}u(\xi^{\prime})\cos^2\frac{v(\xi^{\prime})}{2}q(\xi^{\prime})d\xi^{\prime}\right\|_{L^\infty}\\
			&+\frac{1}{2\sqrt{\mu}}\left(\frac{|A|}{2}+\frac{|s|}{2\mu}+\frac{|B|}{m+1}\|u\|_{L^\infty}^{m-1}+\frac{|s|}{2\mu}\right)E_0.
		\end{align*}
		A crucial step is to bound
		$$\left\|\int_{-\infty}^{\infty}\exp\left\{-\frac{1}{\sqrt{\mu}}\left|\int_{\xi}^{\xi^{\prime}}\cos^{2}\frac{v(s)}{2}\cdot q(s)ds\right|\right\}u(\xi^{\prime})\cos^2\frac{v(\xi^{\prime})}{2}q(\xi^{\prime})d\xi^{\prime}\right\|_{L^\infty}.$$
		This can be achieved by employing variable transformations and contradiction arguments inspired by \cite{luo2021globally}. For this purpose, we introduce the following lemma and its corollary regarding variable transformations.
		\begin{lemm}\cite{Zhou2008RealVariables}
			If $g(t)$ is differentiable on a.e. $[a,b]$, $f(x)\in L^1([c,d])$, and $g([a, b])\subset[c,d]$, then we have that $F(g(t))$ is absolutely
			continuous on $[a,b]$ if and only if $f(g(t))g^{\prime}(t)\in L^{1}([a,b])$ and $\int_{g(a)}^{g(b)}f(x)dx=\int_{a}^{b}f(g(t))g^{\prime}(t)dt$, with $F(t)=\int_{c}^{t}f(x)dx$.
		\end{lemm}
		\begin{coro}\cite{Zhou2008RealVariables}\label{cor}
			Assume that $g(t)$ is absolutely continuous on $[a,b]$, $f(x)\in L^1([c,d])$, and $g([a,b])\subset[c,d]$. If $g(t)$ is monotonous
			or $f(x)\in L^\infty([c,d])$, then we have $\int_{g(a)}^{g(b)}f(x)dx=\int_{a}^{b}f(g(t))g^{\prime}(t)dt$.
		\end{coro}
		Define $y(t,\xi)$ as the solutions of 
		\begin{equation}\label{partial_t y(t,xi)}
			y_t(t,\xi)=\frac{s}{\mu}u(t,\xi)-\frac{\eta}{\mu},\quad y(0,\xi)=y_0(\xi).
		\end{equation}
		Then we have
		\begin{equation}\label{y(t,xi)}
			y(t,\xi)=y_0(\xi)+\int_{0}^{t}\left(\frac{s}{\mu}u(t^{\prime},\xi)-\frac{\eta}{\mu}\right)dt^{\prime}.
		\end{equation}
		It follows from \eqref{y0} that $y_0\in L^\infty_{loc}$ is strictly monotonous and satisfies
		$$|y_0(\xi_1)-y_0(\xi_2)|=\left|\int_{y_0(\xi_2)}^{y_0(\xi_1)}1\mathrm{d}x\right|\leq\left|\int_{y_0(\xi_2)}^{y_0(\xi_1)}1+u_{0x}^2\mathrm{d}x\right|\leq|\xi_1-\xi_2|,$$
		which implies that $y_0$ is local Lipschitz continuous.	Consequently, we see from Step 1 and \eqref{y(t,xi)} that there exists a time $T>0$ such that $y(t,\xi)\in H^1_{loc}$ for $t\in[0,T),$ which means $y(t,\xi)$ is local absolutely continuous for $t\in[0,T)$. Now we show that $T$ can be any positive number.
		
		We claim that 
		\begin{equation}\label{y_xi}
			y_\xi=q\cos^2\frac{v}{2}.
		\end{equation}
		Indeed, \eqref{u_xi} and \eqref{partial_t y(t,xi)} give
		$$\partial_t y_\xi=\frac{s}{\mu}u_\xi=\frac{s}{2\mu}q\sin v.$$
		Furthermore, \eqref{semi-linear system} yields
		$$\partial_t\left(q\cos^2\frac{v}{2}\right)=q_t\cos^2\frac{v}{2}-\frac{1}{2}qv_t\sin v=\frac{s}{2\mu}q\sin v.$$
		Since $(y_\xi-q\cos^2\frac{v}{2})\mid_{t=0}=0$, \eqref{y_xi} holds for $t\in[0,T)$.
		
		Consequently, Corollary \ref{cor} implies that, for $t\in[0,T)$ and $[a,b]\subset\mathbb{R}$, 
		\begin{align}
			&\left\|\int_{a}^{b}\exp\left\{-\frac{1}{\sqrt{\mu}}\left|\int_{\xi}^{\xi^{\prime}}\cos^{2}\frac{v(s)}{2}\cdot q(s)ds\right|\right\}|u(\xi^{\prime})|\cos^2\frac{v(\xi^{\prime})}{2}q(\xi^{\prime})d\xi^{\prime}\right\|_{L^\infty}\notag\\
			\leq&\|u\|_{L^{\infty}}\left\|\int_{a}^{b}\exp\left\{-\frac{1}{\sqrt{\mu}}|y(\xi)-y(\xi^{\prime})|\right\}y_{\xi^{\prime}}d\xi^{\prime}\right\|_{L^{\infty}}\notag\\
			\leq&\|u\|_{L^{\infty}}\int_{-\infty}^{+\infty}e^{-\frac{1}{\sqrt{\mu}}|s|}ds.\label{left hand side}
		\end{align}
		When $a\to-\infty$ and $b\to+\infty$, the monotone convergence theorem guarantees the existence of the limit on the left-hand side of \eqref{left hand side}. Thus, 
		\begin{equation}
			\|P\|_{L^\infty}\leq C\left(E_0^{\frac{1}{2}}+E_0+E_0^{\frac{m+1}{2}}\right).
		\end{equation}
		Similarly, we have
		\begin{equation}\label{bdd_Px_Linfty}
			\|P_x\|_{L^\infty}\leq C\left(E_0^{\frac{1}{2}}+E_0+E_0^{\frac{m+1}{2}}\right).
		\end{equation}
		From \eqref{semi-linear system}, we see that
		$$
		|q_t|\leq C\left(1+E_0^{\frac{1}{2}}+E_0+E_0^{\frac{m+1}{2}}\right)q,
		$$
		which implies
		\begin{equation}\label{bdd_q(t)}
			\exp\left\{-C\left(1+E_0^{\frac{1}{2}}+E_0+E_0^{\frac{m+1}{2}}\right)t\right\}\leq q(t) \leq\exp\left\{C\left(1+E_0^{\frac{1}{2}}+E_0+E_0^{\frac{m+1}{2}}\right)t\right\}.
		\end{equation}
		Therefore, \eqref{y_xi} yields
		$$
		\|y_\xi\|_{L^\infty}\leq\|q\|_{L^\infty}\leq\exp\left\{C\left(1+E_0^{\frac{1}{2}}+E_0+E_0^{\frac{m+1}{2}}\right)T\right\}.
		$$
		In addition, it follows from \eqref{bdd_u} and \eqref{y(t,xi)} that
		\begin{equation}\label{bdd_y(t,xi)}
			y_0(\xi)-C\left(1+E_0^{\frac{1}{2}}\right)t\leq y(t,\xi)\leq y_0(\xi)+C\left(1+E_0^{\frac{1}{2}}\right)t,
		\end{equation}
		which means $y(t,\xi)\in L^\infty_{loc}$ for $t\in[0,T]$. Hence, we have $y(t,\xi)\in H^1_{loc}$ for $t\in[0,T]$. A contradiction argument shows that $T$ cannot be bounded above, which means that the results above hold for all $t\geq0$.
		
		Using \eqref{semi-linear system}, we derive
		$$\|v(t)\|_{L^\infty}\leq B(1+t)$$
		and 
		$$
		\frac{d}{dt}\|u\|_{L^2}^2=\int_{\mathbb{R}}2uu_td\xi\leq2\|u\|_{L^2}\|P_x\|_{L^2},
		$$
		where $B=B(E_0)>0$. Let $\kappa$ be the right-hand side of \eqref{bdd_q(t)} and 
		\begin{equation*}
			\widetilde{\Gamma}(\zeta)\triangleq\min\left\{1,\exp\left[\frac{1}{\sqrt{\mu}}\left(\frac{9E_0}{\mu}-\frac{|\zeta|}{\kappa}\right)\right]\right\}.
		\end{equation*}
		Then we have $\kappa^{-1}\leq q(t)\leq\kappa$ and
		\begin{equation*}
			\|\widetilde{\Gamma}\|_{L^1}\leq 4\kappa\sqrt{\mu}+\frac{36\kappa E_0}{\mu}<\infty.
		\end{equation*}
		Consequently,
		\begin{align*}
			\|P_x\|_{L^2}\leq\frac{1}{2\mu}\|\widetilde{\Gamma}\|_{L^1}\kappa^{\frac{1}{2}}\left(\left|\frac{A}{2}+\frac{s}{2\mu}\right|E_0+\left|\frac{B}{m+1}\right|E_0^{\frac{m+1}{2}}+\left|2k+\frac{\eta}{\mu}\right|E_0^{\frac{1}{2}}+\left|\frac{s}{2}\right|\frac{1}{\sqrt{\mu}}E_0^{\frac{1}{2}}\right)\leq C,
		\end{align*}
		where $C=C(E_0,\kappa)>0$. The estimate for $\|P\|_{L^2}$ is entirely similar. Thus, we obtain
		$$
		\frac{d}{dt}\|u\|_{L^2}^2\leq C(E_0,\kappa)\|u\|_{L^2},
		$$
		which establishes the boundedness of $\|u(t)\|_{L^2}$ for $t$ in bounded intervals. Similarly, 
		\begin{align*}
			\frac{d}{dt}\|v\|_{L^2}^2\leq&2\|v\|_{L^2}\|v_t\|_{L^2}\\
			\leq& C\|v\|_{L^2}\left(\|u\|_{L^2}\|u\|_{L^\infty}+\|u\|_{L^2}\|u\|_{L^\infty}^{m}+\|u\|_{L^2}+\|P\|_{L^2}+\|v\|_{L^2}\|v\|_{L^\infty}\right).
		\end{align*}
		From the previous bounds, we deduce that $\|v(t)\|_{L^2}$ is bounded for $t$ in bounded intervals. It follows from \eqref{u_xi} that
		$$\|u_\xi\|_{L^2}=\|\frac{q}{2}\sin v\|_{L^2}\leq\frac{\kappa}{2}\|v\|_{L^2},$$
		which guarantees the boundedness of $\|u_\xi(t)\|_{L^2}$ on bounded intervals of time. Hence, we complete the proof that the solution of \eqref{semi-linear system}-\eqref{semi-linear system_initial data} can be extended globally in time.
	\end{proof}
	
	For future use, we state an important property of the above solutions, which is the very reason for requiring $s\neq0$.
	
	\begin{lemm}\label{lemma_N}
		Let $s\neq0$, 
		$$\mathcal{N}\triangleq\left\{t\geq0:\mathrm{meas}\{\xi\in \mathbb{R}:v(t,\xi)=-\pi\}>0\right\}.$$
		Then we have
		$$\mathrm{meas}(\mathcal{N})=0.$$
	\end{lemm}
	\begin{proof}
		When $v=-\pi$, we have $1+\cos v=0$ and $v_t=-\frac{s}{\mu}$ by \eqref{semi-linear system}. Given the bounds on $\|u\|_{L^\infty}$ and $\|P\|_{L^\infty}$, there exists a constant $\delta>0$ such that,  whenever $1+\cos v<\delta$, we have 
		$$v_t<-\frac{s}{2\mu}\ \mathrm{for}\ s>0, \quad \mathrm{and}\quad v_t>-\frac{s}{2\mu}\ \mathrm{for}\ s<0.$$
		Since the map $t\mapsto v(t,\xi)$ is absolutely continuous at each fixed $\xi\in\mathbb{R}$, we obtain $v_t=0$ a.e. on $\{v(t,\xi)=-\pi\}$. Hence, a contradiction argument shows that $\mathrm{meas}(\mathcal{N})=0$.
	\end{proof}
	
	\subsection{Global conservative solutions for the original equation}
	In this subsection, we use the global solutions of \eqref{semi-linear system}-\eqref{semi-linear system_initial data} to construct a global conservative solution for the original equation \eqref{eq1}.
	
	\begin{theo}
		Let $s\neq0,\ u_0\in H^1(\mathbb{R})$. Then the Cauchy problem \eqref{eq1} has a global conservative solution in the sense of Definition \ref{conservative}.
	\end{theo}
	\begin{proof}
		Given a global solution $(u,v,q)$ of \eqref{semi-linear system}-\eqref{semi-linear system_initial data}, we set
		\begin{equation}\label{def_u(t,x)}
			u(t,x)\triangleq u(t,\xi),\quad\mathrm{if}\quad x=y(t,\xi),
		\end{equation}
		with $y(t,\xi)$ given by \eqref{partial_t y(t,xi)}. Now we check that \eqref{def_u(t,x)} is well defined. From \eqref{y0} and \eqref{bdd_y(t,xi)}, we see
		$$\lim_{\xi\to\pm\infty}y(t,\xi)=\pm\infty.$$
		Next, \eqref{y_xi} implies that the map $\xi\mapsto y(t,\xi)$ is non-decreasing. If $\xi_1<\xi_2$ but $y(t,\xi_1)=y(t,\xi_2)$, then
		$$\int_{\xi_1}^{\xi_2}q(t,s)\cos^2\frac{v(t,s)}{2}ds=\int_{\xi_1}^{\xi_2}y_\xi(t,s)ds=0,$$
		which yields $\cos\frac{v}{2}=0$ throughout the interval of integration. Thus, \eqref{u_xi} gives
		$$u(t,\xi_2)-u(t,\xi_1)=\int_{\xi_1}^{\xi_2}\frac{q(t,s)}{2}\sin v(t,s)ds=0,$$
		which proves that \eqref{def_u(t,x)} is well defined for all $t\geq0$ and $x\in\mathbb{R}$.
		
		According to the definition \eqref{def_u(t,x)}, we have
		\begin{equation*}
			u_x(t,x)=\frac{\sin v(t,\xi)}{1+\cos v(t,\xi)}\quad\mathrm{if}\quad x=y(t,\xi),\quad\cos v(t,\xi)\neq-1,
		\end{equation*}
		which, together with \eqref{y_xi}, implies
		\begin{align*}
			&\int_{\mathbb{R}}u^{2}(t,x)+\mu u_{x}^{2}(t,x)dx \\
			=&\int_{\{\cos v>-1\}}\left(u^2(t,\xi)\cos^2\frac{v(t,\xi)}{2}+\mu\sin^2\frac{v(t,\xi)}{2}\right)q(t,\xi)d\xi \\
			\leq& E_{0}.
		\end{align*}
		Since $H^1\hookrightarrow C^{0,\frac{1}{2}}$, we see that $u$ is uniformly H\"{o}lder continuous with exponent $\frac{1}{2}$ as a function of $x$. According to \eqref{semi-linear system} and \eqref{bdd_Px_Linfty}, the map $t\mapsto u(t,y(t))$ is uniformly Lipschitz continuous along each characteristic curve $t\mapsto y(t)$. Hence, $u=u(t,x)$ is globally H\"{o}lder continuous.
		
		Now we claim that the map $t\mapsto u(t,\cdot)$ is Lipschitz continuous with values in $L^2(\mathbb{R})$. Indeed, for any interval $[\tau,\tau+h]$ and a given point $x$, we choose $\xi\in\mathbb{R}$ such that the characteristic $t\mapsto y(t,\xi)$ passes through $(\tau,x)$. It follows from \eqref{semi-linear system} and \eqref{bdd_u} that 
		$$\|y_t\|_{L^\infty}\leq\frac{|s|E_0^{\frac{1}{2}}}{\mu^{\frac{5}{4}}}+\frac{|\eta|}{\mu}\triangleq\widetilde{C}$$
		and 
		\begin{align*}
			&\left|u(\tau+h,x)-u(\tau,x)\right|\\
			\leq&\left|u(\tau+h,x)-u\left(\tau+h,y(\tau+h,\xi)\right)\right|  +\left|u\left(\tau+h,y(\tau+h,\xi)\right)-u(\tau,x)\right| \\
			\leq&\sup_{|y-x|\leq \widetilde{C}h}\left|u(\tau+h,y)-u(\tau+h,x)\right|+\int_{\tau}^{\tau+h}\left|P_{x}(t,\xi)\right|dt.
		\end{align*}
		Integrating over $\mathbb{R}$ yields
		\begin{align*}
			\int_\mathbb{R}\left|u(\tau+h,x)-u(\tau,x)\right|^2dx\leq&2\int_\mathbb{R}\left(\int_{x-\widetilde{C}h}^{x+\widetilde{C}h}\left|u_x(\tau+h,y)\right|dy\right)^2dx\\
			&+2\int_\mathbb{R}\left(\int_\tau^{\tau+h}\left|P_x(t,\xi)\right|dt\right)^2q(\tau,\xi)\cos^2\frac{v(\tau,\xi)}{2}d\xi\\
			\leq&8\widetilde{C}^2h^2\|u_x(\tau+h)\|_{L^2}^2+2h\|q(\tau)\|_{L^\infty}\int_\tau^{\tau+h}\|P_x(t)\|_{L^2}^2dt\\
			\leq&Ch^2,
		\end{align*}
		thus proving the claim.
		
		Since $L^2(\mathbb{R})$ is a reflexive space, the left-hand side of \eqref{u} is well defined for a.e. $t\geq0$ and the right-hand side of \eqref{u} also lies in $L^2(\mathbb{R})$ for a.e. $t\geq0$. From \eqref{semi-linear system}, we observe that for any $t>0$ and $\xi\in\mathbb{R}$,
		$$\frac{d}{dt}u\left(t,y(t,\xi)\right)=-P_x(t,\xi),$$
		where $P_x$ is the function defined at \eqref{Px}. Moreover, for every $t\notin\mathcal{N}$, the map $\xi\mapsto y(t,\xi)$ is one-to-one and a
		change of variables yields
		\begin{align*}
			P_x(t,\xi) 
			=&\frac{1}{2\mu}\left(\int_{y(t,\xi)}^{\infty}-\int_{-\infty}^{y(t,\xi)}\right)\exp\left\{-\frac{1}{\sqrt{\mu}}\left|y(t,\xi)-x\right|\right\}\\
			&\cdot\left(-\left(\frac{A}{2}+\frac{s}{2\mu}\right)u^2+\frac{s}{2}u_x^2-\frac{B}{m+1}u^{m+1}+\left(2k+\frac{\eta}{\mu}\right)u\right)(t,x)dx\\
			=&P_x(t,y(t,\xi)).
		\end{align*}
		Since Lemma \ref{lemma_N} guarantees $\mathrm{meas}(\mathcal{N})=0$, the identity \eqref{u} holds for almost every $t\geq0$.
		
		Finally, we consider the balance law. Let $\left\{\mu_{(t)},\ t\geq0\right\}$ be the bounded Radon measures defined by
		\begin{equation}
			\mu_{(t)}\left([a,b]\right)\triangleq\int_{\{\xi;y(t,\xi)\in[a,b]\}}\sin^2\frac{v(t,\xi)}{2}\cdot q(t,\xi)d\xi.
		\end{equation}
		If $t\notin\mathcal{N}$, then the measure $\mu_{(t)}$ is absolutely continuous, possessing density $u_x^2$ with respect to the Lebesgue measure. Moreover, for $t\in\mathcal{N}$, $\mu_{(t)}$ is the weak limit of $\mu_{(t^\prime)}$, as $t^\prime\to t,\ t^\prime\notin\mathcal{N}$. It follows from \eqref{semi-linear system} and \eqref{u_xi} that for any test function $\varphi\in\mathcal{C}_c^1(\mathbb{R}^+\times\mathbb{R})$,
		\begin{align*}
			&-\int_{0}^{\infty}\int_\mathbb{R}\varphi_t+\varphi_x\left(\frac{s}{\mu}u-\frac{\eta}{\mu}\right)d\mu_{(t)}dt\\
			=&-\int_{0}^{\infty}\int_\mathbb{R}\partial_t(\varphi(t,y(t,\xi)))\cdot\sin^2\frac{v(t,\xi)}{2}\cdot q(t,\xi)d\xi dt\\
			=&\frac{1}{\mu}\int_{0}^{\infty}\int_\mathbb{R}\varphi(t,y(t,\xi))\cdot\left[\left(-\left(\frac{A}{2}+\frac{s}{2\mu}\right)u^2-\frac{B}{m+1}u^{m+1}+\left(2k+\frac{\eta}{\mu}\right)u-P\right)\sin v\cdot q\right](t,\xi) d\xi dt\\
			&+\int_{\mathbb{R}}\varphi(0,y(0,\xi))\cdot\sin^2\frac{v(0,\xi)}{2}\cdot q(0,\xi)d\xi\\
			=&\frac{2}{\mu}\int_{0}^{\infty}\int_\mathbb{R}\left(-\left(\frac{A}{2}+\frac{s}{2\mu}\right)u^2-\frac{B}{m+1}u^{m+1}+\left(2k+\frac{\eta}{\mu}\right)u-P\right)u_x\varphi  dxdt+\int_{\mathbb{R}}\varphi(0,x)u_{0x}^{2}dx.
		\end{align*}	
		%Since Lemma \ref{lemma_N} gives $\mathrm{meas}(\mathcal{N})=0$, \eqref{ux2} yields the balance law \eqref{eqw}, which can be formulated more precisely as
		
		Hence, we complete the proof that $u$ is a global conservative solution of \eqref{eq1} in the sense of Definition \ref{conservative}.
	\end{proof}
	\section{Uniqueness}\label{Uniqueness}
	In this section, we establish the following theorem on the uniqueness of global conservative solutions for the Cauchy problem \eqref{eq1}.
	
	\begin{theo}\label{thm_uniqueness}
		Let $u(t,x)$ be a global conservative solution of \eqref{eq1} with $s\neq0$ in the sense of Definition \ref{conservative}. Then $u(t,x)$ is unique.
	\end{theo}
	\subsection{Preliminary lemmas}
	Let $u=u(t,x)$ be a global conservative solution of \eqref{eq1} with $s\neq0$ in the sense of Definition \ref{conservative}. Let $y$ still denote the solutions to
	\begin{equation}\label{partial_t y(t)}
		\frac{d}{dt}y(t)=\frac{s}{\mu}u(t,y(t))-\frac{\eta}{\mu},\quad y(0)=y_0.
	\end{equation}
	Since $u$ is only H\"{o}lder continuous, uniqueness of solutions to \eqref{partial_t y(t)} is not guaranteed. However, the balance law \eqref{eqw} and \eqref{partial_t y(t)} yield
	\begin{equation}\label{coupled with (4.1)}
		\frac{d}{dt}\int_{-\infty}^{y(t)}d\mu_{(t)}=\frac{2}{\mu}\int_{-\infty}^{y(t)}u_x\left(-\left(\frac{A}{2}+\frac{s}{2\mu}\right)u^2-\frac{B}{m+1}u^{m+1}+\left(2k+\frac{\eta}{\mu}\right)u-P\right)(t,x) dx.
	\end{equation}
	In the following, we will prove that the solution to both \eqref{partial_t y(t)} and \eqref{coupled with (4.1)} is unique.
	
	Instead of the variables $(t,x)$, it is convenient to work with an adapted set of variables $(t,\beta)$. At any time $t$ and $\beta\in\mathbb{R}$, we define $y(t,\beta)$ to be the unique $y$ such that
	\begin{equation}
		y(t,\beta)+\mu_{(t)}\{(-\infty,y)\}\leq\beta\leq y(t,\beta)+\mu_{(t)}\{(-\infty,y]\}.
	\end{equation}
	When $\mu_{(t)}$ is absolutely continuous with density $u_x^2$ with respect to the Lebesgue measure, the above definition shows that
	\begin{equation}\label{y(t,beta)_absolutely continuous}
		y(t,\beta)+\int_{-\infty}^{y(t,\beta)}u_x^2(t,z)dz=\beta.
	\end{equation}
	
	We now establish the Lipschitz continuity of $y$ and $u$ as functions of $t$ and $\beta$ through a series of lemmas.
	\begin{lemm}\label{lip_3}
		Let $u=u(t,x)$ be a conservative solution of \eqref{eq1} with $s\neq0$. Then, for any $t\geq0$, the maps $\beta\mapsto y(t,\beta)$ and $\beta\mapsto u(t,\beta)\triangleq u(t,y(t,\beta))$ are Lipschitz continuous with constant 1. Moreover, the map $t\mapsto y(t,\beta)$ is also Lipschitz
		continuous with a constant depending only on $\|u_0\|_{H^1}$.
	\end{lemm}
	\begin{proof}
		We carry out the proof in three steps.
		
		\textbf{Step 1.} For any fixed time $t\geq0$, the map
		$$y\mapsto F_t(y)\triangleq y+\mu_{(t)}\{(-\infty,y]\}$$
		is right continuous and strictly increasing. Hence it has a well defined, continuous, non-decreasing inverse $\beta\mapsto y(t,\beta)$. If $\beta_1<\beta_2$, then 
		$$y(t,\beta_2)-y(t,\beta_1)+\mu_{(t)}\left\{(y(t,\beta_1),y(t,\beta_2))\right\}\leq\beta_2-\beta_1,$$
		which yields
		$$y(t,\beta_2)-y(t,\beta_1)\leq\beta_2-\beta_1.$$
		
		\textbf{Step 2.} Let $\beta_1<\beta_2$. Then we have 
		\begin{align}
			\left|u(t,y(t,\beta_{2}))-u(t,y(t,\beta_{1}))\right|\leq\int_{y(t,\beta_{1})}^{y(t,\beta_{2})}|u_{x}|dx\leq\int_{y(t,\beta_{1})}^{y(t,\beta_{2})}\frac{1}{2}(1+u_{x}^{2})dx\notag\\
			\leq\frac{1}{2}\left[y(t,\beta_{2})-y(t,\beta_{1})+\mu_{(t)}\left\{(y(t,\beta_{1}),y(t,\beta_{2}))\right\}\right]\leq\frac{1}{2}(\beta_{2}-\beta_{1}).\label{lip_u_beta}
		\end{align}
		
		\textbf{Step 3.} Assume $y(\tau,\beta)=\bar{y}$. Since
		$$
		\|\frac{s}{\mu}u-\frac{\eta}{\mu}\|_{L^\infty}\leq C(\|u\|_{H^1}+1)\leq C_{\infty}
		$$
		and 
		\begin{align*}
			&\|\frac{2}{\mu}u_x\left(-\left(\frac A2+\frac{s}{2\mu}\right)u^2-\frac{B}{m+1}u^{m+1}+\left(2k+\frac{\eta}{\mu}\right)u-P\right)\|_{L^1}\\
			\leq&C\|u_x\|_{L^2}\left(\|u\|_{L^\infty}\|u\|_{L^2}+\|u\|_{L^\infty}^m\|u\|_{L^2}+\|u\|_{L^2}+\|P\|_{L^2}\right)\\
			\leq&C_S
		\end{align*}
		for $C_\infty,\ C_S>0$ depending only on $\|u_0\|_{H^1}$, the balance law \eqref{eqw} implies that for $t>\tau$, 
		\begin{align*}
			&\mu_{(t)}\{ (-\infty,\bar{y}-C_{\infty}(t-\tau))\}\\
			\leq& \mu_{(\tau)}\{(-\infty,\bar{y})\}+\int_{\tau}^{t}\left\|\frac{2}{\mu}u_x\left(-\left(\frac A2+\frac{s}{2\mu}\right)u^2-\frac{B}{m+1}u^{m+1}+\left(2k+\frac{\eta}{\mu}\right)u-P\right)\right\|_{L^{1}(\mathbb{R})} dt \\
			\leq& \mu_{(\tau)}\{(-\infty, \bar{y})\}+C_{S}(t-\tau) .
		\end{align*}
		Defining $y^{-}(t)\triangleq \bar{y}-(C_{\infty}+C_{S})(t-\tau)$, we have
		\begin{align*}
			& y^{-}(t)+\mu_{(t)}\{(-\infty , y^{-}(t))\} \\
			\leq& \bar{y}-(C_{\infty}+C_{S})(t-\tau)+\mu_{(\tau)}\{(-\infty , \bar{y})\}+C_{S}(t-\tau) \\
			\leq& \bar{y}+\mu_{(\tau)}\{(-\infty ,\bar{y})\} \leq \beta .
		\end{align*}
		which leads to $y(t,\beta)\geq y^{-}(t)$ for all $t>\tau$. Similarly, we have $y(t,\beta)\leq y^{+}(t)\triangleq \bar{y}+(C_{\infty}+C_{S})(t-\tau)$. Hence, $t\mapsto y(t,\beta)$ is  Lipschitz continuous with a constant depending only on $\|u_0\|_{H^1}$.		
	\end{proof}
	
	\begin{lemm}
		Let $u=u(t,x)$ be a conservative solution of \eqref{eq1} with $s\neq0$. Then, for any $y_0\in\mathbb{R}$, there exists a unique Lipschitz continuous map $t\mapsto y(t)$ satisfying both \eqref{partial_t y(t)} and \eqref{coupled with (4.1)}. Furthermore, for any $0\leq\tau\leq t$, we have
		\begin{equation}\label{u_Px}
			u(t,y(t))-u(\tau,y(\tau))=-\int_\tau^t P_x(s,y(s))ds.
		\end{equation}
	\end{lemm}
	\begin{proof}
		We divide the proof into six steps.
		
		\textbf{Step 1.} In the adapted coordinates $(t,\beta)$, we express the characteristic starting at $y_0$ as $t \mapsto y(t) = y(t, \beta(t))$, where $\beta(\cdot)$ is a map to be determined later. Combining \eqref{partial_t y(t)} with \eqref{coupled with (4.1)} and
		integrating with respect to time, we obtain
		\begin{align}
			&y(t)+\int_{-\infty}^{y(t)}d\mu_{(t)}=y_0+\int_{-\infty}^{y_0}u_{0x}^2(x)dx+\int_{0}^{t}\left(\frac{s}{\mu}u(t^\prime,y(t^\prime))-\frac{\eta}{\mu}\right)dt^\prime\notag\\
			&+\frac{2}{\mu}\int_{0}^{t}\int_{-\infty}^{y(t^\prime)}u_x\left(-\left(\frac{A}{2}+\frac{s}{2\mu}\right)u^2-\frac{B}{m+1}u^{m+1}+\left(2k+\frac{\eta}{\mu}\right)u-P\right)(t^\prime,x)dxdt^\prime.\label{combined}
		\end{align}
		Introducing the function
		\begin{align}\label{G(t,beta)}
			G(t,\beta)\triangleq\int_{-\infty}^{y(t,\beta)}\frac{s}{\mu}u_x+\frac{2}{\mu}u_x\left(-\left(\frac{A}{2}+\frac{s}{2\mu}\right)u^2-\frac{B}{m+1}u^{m+1}+\left(2k+\frac{\eta}{\mu}\right)u-P\right)dx-\frac{\eta}{\mu}
		\end{align}
		and the constant
		\begin{equation}\label{beta_0}
			\beta_0\triangleq y_0+\int_{-\infty}^{y_0}u_{0x}^2(x)dx,
		\end{equation}
		we can rewrite \eqref{combined} in the form
		\begin{equation}\label{beta(t)}
			\beta(t)=\beta_0+\int_0^tG(t^\prime,\beta(t^\prime))dt^\prime.
		\end{equation}
		
		\textbf{Step 2.} For each fixed $t\geq 0$, since $u(t,\cdot)$ and $P(t,\cdot)$ are both in $H^1(\mathbb{R})$, the function $\beta\mapsto G(t,\beta)$ is uniformly bounded and absolutely continuous. Hence, it follows from \eqref{y(t,beta)_absolutely continuous} and \eqref{G(t,beta)} that
		\begin{equation}\label{G(t,beta)_lip}
			|G(t,\beta_{2})-G(t,\beta_{1})|\leq C|\beta_{2}-\beta_{1}|,\quad \forall\ \beta_{1},\beta_{2}\in\mathbb{R},
		\end{equation}
		where $C>0$ is a constant depending only on $\|u_0\|_{H^1}$.
		
		\textbf{Step 3.} The uniform Lipschitz continuity of $G$ allows us to establish, via a standard fixed-point argument, the existence of a unique solution to the integral equation \eqref{beta(t)}. Consider the Banach space of all continuous functions $\beta:\mathbb{R}^+\to\mathbb{R}$ with weighted norm
		$$\|\beta\|_\star\triangleq\sup_{t\geq0}e^{-2Ct}|\beta(t)|.$$
		On this space, we assert that the Picard map
		$$(\mathcal{P}\beta)(t)\triangleq \beta_0+\int_0^tG(\tau,\beta(\tau))d\tau$$
		is a strict contraction. In fact, if $\|\beta-\tilde{\beta}\|_*=\delta>0$, then we have $|\beta(\tau)-\tilde{\beta}(\tau)|\leq\delta e^{2C\tau}$ for all $\tau\geq0$. Therefore, 
		\begin{align*}
			\left|(\mathcal{P}\beta)(t)-(\mathcal{P}\tilde{\beta})(t)\right| & =\left|\int_{0}^{t}\left(G(\tau,\beta(\tau))-G(\tau,\tilde{\beta}(\tau))\right)d\tau\right| \\
			& \leq C\int_{0}^{t}|\beta(\tau)-\tilde{\beta}(\tau)|d\tau \\
			& \leq C\delta\int_{0}^{t} e^{2C\tau}d\tau\leq\frac{\delta}{2}e^{2Ct},
		\end{align*}
		which implies $\|\mathcal{P}\beta-\mathcal{P}\tilde{\beta}\|_{*}\leq\delta/2$. The contraction mapping principle thus yields a unique solution to \eqref{beta(t)}.
		
		\textbf{Step 4.} The above analysis yields a solution to \eqref{combined} in the form $t \mapsto y(t) \triangleq y(t, \beta(t))$. It follows from Lemma \ref{lip_3} and the Lipschitz continuity of $\beta(t)$ that $t \mapsto y(t, \beta(t))$ is Lipschitz continuous, which means that $y(t)$ is differentiable almost everywhere. To prove that $y(t)$ satisfies \eqref{partial_t y(t)}, it suffices to verify the equation at those times $\tau > 0$ for which $y(\cdot)$ is differentiable at $t=\tau$.
		%(2) the measure $\mu_{(\tau)}$ is absolutely continuous.
		
		Assume, on the contrary, that $y(\cdot)$ is differentiable at $t=\tau$ but $\dot{y}(\tau)\neq \frac{s}{\mu}u(\tau,y(\tau))-\frac{\eta}{\mu}$. Without loss of generality, let
		\begin{equation}\label{dot_y_contrary}
			\dot{y}(\tau)=\frac{s}{\mu}u(\tau,y(\tau))-\frac{\eta}{\mu}+2\varepsilon_0
		\end{equation}
		for some $\varepsilon_0>0$. The situation for $\varepsilon_0<0$ is completely similar. Then, for $t\in[\tau,\tau+\delta]$, with $\delta>0$ small enough, we have
		\begin{equation}
			y^+(t)\triangleq y(\tau)+(t-\tau)\left[\frac{s}{\mu}u(\tau,y(\tau))-\frac{\eta}{\mu}+\varepsilon_0\right]<y(t).
		\end{equation}
		By an approximation argument, we can extend the validity of \eqref{w} to any test function $\varphi \in H^1$ that is Lipschitz continuous with compact support. For any $\varepsilon>0$ small enough, we thus consider the functions
		\begin{equation*}
			\rho^\varepsilon(s,x)\triangleq\left\{
			\begin{array}{cll}
				0 & \quad\mathrm{if}\quad x\leq-\varepsilon^{-1}, \\
				x+\varepsilon^{-1} & \quad\mathrm{if}\quad-\varepsilon^{-1}\leq x\leq1-\varepsilon^{-1}, \\
				1 & \quad\mathrm{if}\quad1-\varepsilon^{-1}\leq x\leq y^+(s), \\
				1-\varepsilon^{-1}(x-y^+(s)) & \quad\mathrm{if}\quad y^+(s)\leq x\leq y^+(s)+\varepsilon, \\
				0 & \quad\mathrm{if}\quad x\geq y^+(s)+\varepsilon,
			\end{array}\right.
		\end{equation*}
		\begin{equation}\label{chi_epsilon}
			\chi^\varepsilon(s)\triangleq\left\{
			\begin{array}{cll}
				0 & \quad\mathrm{if}\quad s\leq\tau-\varepsilon, \\
				\varepsilon^{-1}(s-\tau+\varepsilon) & \quad\mathrm{if}\quad\tau-\varepsilon\leq s\leq\tau, \\
				1 & \quad\mathrm{if}\quad\tau\leq s\leq t, \\
				1-\varepsilon^{-1}(s-t) & \quad\mathrm{if}\quad t\leq s\leq t+\varepsilon, \\
				0 & \quad\mathrm{if}\quad s\geq t+\varepsilon.
			\end{array}\right.
		\end{equation}
		Define 
		$$\varphi^\varepsilon(s,x)\triangleq\min\{\rho^\varepsilon(s,x),\chi^\varepsilon(s)\}.$$
		%Since $\mu_{(t)}$ is absolutely continuous with respect to the Lebesgue measure for $t\notin\mathcal{N}$ with density $u_x^2$, and $\mathrm{meas}(\mathcal{N})=0$, 
		Substituting $\varphi^\varepsilon$ as a test function into \eqref{w} yields
		\begin{align}
			&\int_{0}^{\infty}\int_\mathbb{R}\varphi^\varepsilon_t+\varphi^\varepsilon_x\left(\frac{s}{\mu}u-\frac{\eta}{\mu}\right)d\mu_{(t)}dt\notag\\
			+&\frac{2}{\mu}\int_{0}^{\infty}\int_\mathbb{R}\left(-\left(\frac{A}{2}+\frac{s}{2\mu}\right)u^2-\frac{B}{m+1}u^{m+1}+\left(2k+\frac{\eta}{\mu}\right)u-P\right)u_x\varphi^\varepsilon dxdt=0.\label{test_phi_epsilon}
		\end{align}
		If $t$ is sufficiently close to $\tau$, then
		\begin{equation*}
			\lim_{\varepsilon\to0}\int_{\tau}^{t}\int_{y^{+}(t^\prime)}^{y^{+}(t^\prime)+\varepsilon}\varphi^\varepsilon_t+\varphi^\varepsilon_x\left(\frac{s}{\mu}u-\frac{\eta}{\mu}\right)d\mu_{(t^\prime)}dt^\prime\geq0.
		\end{equation*}
		Indeed, for $t^\prime\in(\tau,t)$, $x\in(y^{+}(t^\prime),y^{+}(t^\prime)+\varepsilon)$, the facts that $\frac{s}{\mu}u(t^\prime,x)<\frac{s}{\mu}u(\tau,y(\tau))+\varepsilon_0$ and $\varphi^\varepsilon_x\leq0$ imply
		$$
		0=\varphi^\varepsilon_t+\varphi^\varepsilon_x\left(\frac{s}{\mu}u(\tau,y(\tau))-\frac{\eta}{\mu}+\varepsilon_0\right)\leq \varphi^\varepsilon_t+\varphi^\varepsilon_x\left(\frac{s}{\mu}u(t^\prime,x)-\frac{\eta}{\mu}\right).
		$$
		Owing to the weak continuity of the map $t \mapsto \mu_{(t)}$, taking the limit of \eqref{test_phi_epsilon} as $\varepsilon\to0$ gives
		\begin{align*}
			0=
			&\int_{-\infty}^{y(\tau)}d\mu_{(\tau)}-\int_{-\infty}^{y^{+}(t)}d\mu_{(t)} +\lim_{\varepsilon\to0}\int_{\tau}^{t}\int_{y^{+}(t^\prime)}^{y^{+}(t^\prime)+\varepsilon}\varphi^\varepsilon_t+\varphi^\varepsilon_x\left(\frac{s}{\mu}u-\frac{\eta}{\mu}\right)d\mu_{(t^\prime)}dt^\prime \\
			& 
			+\frac{2}{\mu}\int_{\tau}^{t}\int_{-\infty}^{y^{+}(t^\prime)}\left(-\left(\frac{A}{2}+\frac{s}{2\mu}\right)u^2-\frac{B}{m+1}u^{m+1}+\left(2k+\frac{\eta}{\mu}\right)u-P\right)u_xdxdt^\prime \\
			\geq&\int_{-\infty}^{y(\tau)}d\mu_{(\tau)}-\int_{-\infty}^{y^{+}(t)}d\mu_{(t)}+o_1(t-\tau)\\
			& 
			+\frac{2}{\mu}\int_{\tau}^{t}\int_{-\infty}^{y(t^\prime)}\left(-\left(\frac{A}{2}+\frac{s}{2\mu}\right)u^2-\frac{B}{m+1}u^{m+1}+\left(2k+\frac{\eta}{\mu}\right)u-P\right)u_xdxdt^\prime,
		\end{align*}
		where $o_1(t-\tau)$ satisfies $\frac{o_{1}(t-\tau)}{t-\tau}\to0$ as $t\to\tau$. Indeed, 
		\begin{align*}
			|o_{1}(t-\tau)|=&\frac{2}{\mu}\left|\int_{\tau}^{t}\int_{y^{+}(t^\prime)}^{y(t^\prime)}\left(-\left(\frac{A}{2}+\frac{s}{2\mu}\right)u^2-\frac{B}{m+1}u^{m+1}+\left(2k+\frac{\eta}{\mu}\right)u-P\right)u_xdxdt^\prime\right|\\
			\leq&\frac{2}{\mu}\left\|-\left(\frac{A}{2}+\frac{s}{2\mu}\right)u^2-\frac{B}{m+1}u^{m+1}+\left(2k+\frac{\eta}{\mu}\right)u-P\right\|_{L^\infty}\int_{\tau}^{t}\int_{y^{+}(t^\prime)}^{y(t^\prime)}|u_{x}|dxdt^\prime \\
			\leq&C\int_{\tau}^{t}(y(t^\prime)-y^{+}(t^\prime))^{1/2}\|u_{x}(t^\prime,\cdot)\|_{L^{2}}dt^\prime\\
			\leq&C(t-\tau)^{3/2}.
		\end{align*}
		Therefore, for every $t>\tau$ sufficiently close to $\tau$, we obtain
		\begin{align}
			\beta(t) & =y(t)+\int_{-\infty}^{y(t)}d\mu_{(t)}\notag \\
			& > y(\tau)+(t-\tau)\left[\frac{s}{\mu}u(\tau,y(\tau))-\frac{\eta}{\mu}+\varepsilon_0\right]+\int_{-\infty}^{y^{+}(t)}d\mu_{(t)}\notag \\
			& \geq y(\tau)+(t-\tau)\left[\frac{s}{\mu}u(\tau,y(\tau))-\frac{\eta}{\mu}+\varepsilon_0\right]+\int_{-\infty}^{y(\tau)}d\mu_{(\tau)}+o_1(t-\tau)\notag\\
			&+\frac{2}{\mu}\int_{\tau}^{t}\int_{-\infty}^{y(t^\prime)}\left(-\left(\frac{A}{2}+\frac{s}{2\mu}\right)u^2-\frac{B}{m+1}u^{m+1}+\left(2k+\frac{\eta}{\mu}\right)u-P\right)u_xdxdt^\prime.\label{1!}
		\end{align}
		On the other hand, it follows from \eqref{G(t,beta)} and \eqref{beta(t)} that
		\begin{align}
			\beta(t)=&\beta(\tau)-\frac{2}{\mu}\int_{\tau}^{t}\int_{-\infty}^{y(\tau)}u_xPdxdt^\prime+(t-\tau)\left(\frac{s}{\mu}u(\tau,y(\tau))-\frac{\eta}{\mu}\right)\notag\\
			&-(t-\tau)\left[\frac{2}{3\mu}\left(\frac{A}{2}+\frac{s}{2\mu}\right)u^3+\frac{2B}{\mu(m+1)(m+2)}u^{m+2}-\frac{1}{\mu}\left(2k+\frac{\eta}{\mu}\right)u^{2}\right](\tau,y(\tau))\notag\\
			&+o_2(t-\tau),\label{2!}
		\end{align}
		with $\frac{o_{2}(t-\tau)}{t-\tau}\to0$ as $t\to\tau$. We then derive from \eqref{1!} and \eqref{2!} that
		\begin{align*}
			&\beta(\tau)-\frac{2}{\mu}\int_{\tau}^{t}\int_{-\infty}^{y(\tau)}u_xPdxdt^\prime+(t-\tau)\left(\frac{s}{\mu}u(\tau,y(\tau))-\frac{\eta}{\mu}\right)+o_2(t-\tau)\\
			&-(t-\tau)\left[\frac{2}{3\mu}\left(\frac{A}{2}+\frac{s}{2\mu}\right)u^3+\frac{2B}{\mu(m+1)(m+2)}u^{m+2}-\frac{1}{\mu}\left(2k+\frac{\eta}{\mu}\right)u^{2}\right](\tau,y(\tau))\\
			\geq&y(\tau)+(t-\tau)\left[\frac{s}{\mu}u(\tau,y(\tau))-\frac{\eta}{\mu}+\varepsilon_0\right]+\int_{-\infty}^{y(\tau)}d\mu_{(\tau)}+o_1(t-\tau)\\
			&+\frac{2}{\mu}\int_{\tau}^{t}\int_{-\infty}^{y(t^\prime)}\left(-\left(\frac{A}{2}+\frac{s}{2\mu}\right)u^2-\frac{B}{m+1}u^{m+1}+\left(2k+\frac{\eta}{\mu}\right)u-P\right)u_xdxdt^\prime.
		\end{align*}
		Subtracting common terms, dividing both sides by $t-\tau$ and letting $t\to\tau$, we have $\varepsilon_0\leq0$, which contradicts the fact that $\varepsilon_0>0$. Hence, \eqref{partial_t y(t)} must hold.
		
		\textbf{Step 5.} Our goal now is to prove \eqref{u_Px}. For any test function $\phi\in C_c^\infty(\mathbb{R}^+\times\mathbb{R})$, it follows from \eqref{u} that 
		$$\int_0^\infty\int_\mathbb{R}\left[u\phi_t+\left(\frac{s}{2\mu}u^2-\frac{\eta}{\mu}u\right)\phi_x-P_x\phi\right]dxdt+\int_\mathbb{R} u_0(x)\phi(0,x)dx=0.$$
		Given any $\varphi\in C_c^\infty$, let $\phi=\varphi_x$. Since the map $x\mapsto u(t,x)$ is absolutely continuous, integrating by parts with respect to $x$ yields
		\begin{align}
			\int_0^\infty\int_\mathbb{R}\left[u_x\varphi_t+\left(\frac{s}{\mu}u-\frac{\eta}{\mu}\right)u_x\varphi_x+P_x\varphi_x\right]dxdt+\int_\mathbb{R} u_{0x}(x)\varphi(0,x)dx=0.\label{test_phi2}
		\end{align}
		An approximation argument shows that \eqref{test_phi2} holds for every compactly supported Lipschitz continuous test function $\varphi$. For any $\varepsilon>0$ small enough, we thus consider the functions
		\begin{equation*}
			\varrho^\varepsilon(s,x)\triangleq\left\{
			\begin{array}{cll}
				0 & \quad\mathrm{if}\quad x\leq-\varepsilon^{-1}, \\
				x+\varepsilon^{-1} & \quad\mathrm{if}\quad-\varepsilon^{-1}\leq x\leq1-\varepsilon^{-1}, \\
				1 & \quad\mathrm{if}\quad1-\varepsilon^{-1}\leq x\leq y(s), \\
				1-\varepsilon^{-1}(x-y(s)) & \quad\mathrm{if}\quad y(s)\leq x\leq y(s)+\varepsilon, \\
				0 & \quad\mathrm{if}\quad x\geq y(s)+\varepsilon,
			\end{array}\right.
		\end{equation*}
		and 
		$$\psi^\varepsilon(s,x)\triangleq\min\{\varrho^\varepsilon(s,x),\chi^\varepsilon(s)\},$$
		with $\chi^\varepsilon(s)$ as in \eqref{chi_epsilon}. Since $P_x$ is continuous, substituting $\psi^\varepsilon$ as a test function into \eqref{test_phi2} and letting $\varepsilon\to0$ gives
		\begin{align}
			\int_{-\infty}^{y(t)}u_{x}(t,x)dx=&\int_{-\infty}^{y(\tau)}u_{x}(\tau,x)dx-\int_{\tau}^{t}P_{x}(t^\prime,y(t^\prime))dt^\prime\notag \\
			&+\lim_{\varepsilon\to0}\int_{\tau-\varepsilon}^{t+\varepsilon}\int_{y(t^\prime)}^{y(t^\prime)+\varepsilon}u_{x}\left(\psi_{t}^{\varepsilon}+\left(\frac{s}{\mu}u-\frac{\eta}{\mu}\right)\psi_{x}^{\varepsilon}\right)dxdt^\prime.\label{aaa}
		\end{align}
		To obtain \eqref{u_Px}, we are left to show that the last term on the right-hand side of \eqref{aaa} vanishes. Since $u_x\in L^2$, Cauchy’s inequality leads to
		\begin{align}
			&\left|\int_{\tau}^{t}\int_{y(t^\prime)}^{y(t^\prime)+\varepsilon}u_{x}\left(\psi_{t}^{\varepsilon}+\left(\frac{s}{\mu}u-\frac{\eta}{\mu}\right)\psi_{x}^{\varepsilon}\right)dxdt^\prime\right|\notag\\
			\leq&\int_{\tau}^{t}\left(\int_{y(t^\prime)}^{y(t^\prime)+\varepsilon}u_{x}^2dx\right)^{\frac{1}{2}}\left(\int_{y(t^\prime)}^{y(t^\prime)+\varepsilon}\left(\psi_{t}^{\varepsilon}+\left(\frac{s}{\mu}u-\frac{\eta}{\mu}\right)\psi_{x}^{\varepsilon}\right)^2dx\right)^{\frac{1}{2}}dt^\prime.\label{bbb}
		\end{align}
		Define
		$$\eta_\varepsilon(t)\triangleq\left(\sup_{x\in\mathbb{R}}\int_{x}^{x+\varepsilon}u_x^2(t,y)dy\right)^{\frac{1}{2}}.$$		
		Owing to the uniform boundedness of $\{\eta_{\varepsilon}\}_{\varepsilon>0}$ and the pointwise convergence $\lim\limits_{\varepsilon\to0}\eta_{\varepsilon}(t)=0$ for a.e. $t$, we deduce by the dominated convergence theorem that 
		\begin{align}\label{ccc}
			\lim_{\varepsilon\to0}\int_{\tau}^{t}\left(\int_{y(t^\prime)}^{y(t^\prime)+\varepsilon}u_{x}^2(t^\prime,x)dx\right)^{\frac{1}{2}}dt^\prime\leq\lim_{\varepsilon\to0}\int_{\tau}^{t}\eta_\varepsilon(t^\prime)dt^\prime=0.
		\end{align}
		Furthermore, the definition of $\psi^\varepsilon$ gives that, for $t^\prime\in[\tau,t]$,
		\begin{align}
			&\int_{y(t^\prime)}^{y(t^\prime)+\varepsilon}\left(\psi_{t}^{\varepsilon}+\left(\frac{s}{\mu}u-\frac{\eta}{\mu}\right)\psi_{x}^{\varepsilon}\right)^2(t^\prime,x)dx=\left(\frac{s}{\mu}\right)^2\varepsilon^{-2}\int_{y(t^\prime)}^{y(t^\prime)+\varepsilon}\left|u(t^\prime,x)-u(t^\prime,y(t^\prime))\right|^2dx\notag\\
			\leq&\left(\frac{s}{\mu}\right)^2\varepsilon^{-1}\left(\max_{y(t^\prime)\leq x\leq y(t^\prime)+\varepsilon}|u(t^\prime,x)-u(t^\prime,y(t^\prime))|\right)^{2}\leq\left(\frac{s}{\mu}\right)^2\varepsilon^{-1}\left(\int_{y(t^\prime)}^{y(t^\prime)+\varepsilon}|u_{x}(t^\prime,x)|dx\right)^{2}\notag\\
			\leq&\left(\frac{s}{\mu}\right)^2\varepsilon^{-1}\left(\varepsilon^{\frac{1}{2}}\|u_{x}(t^\prime)\|_{L^{2}}\right)^{2}\leq\left(\frac{s}{\mu}\right)^2\|u(t^\prime)\|_{H^{1}}^{2}.\label{ddd}
		\end{align}
		It follows from \eqref{ccc} and \eqref{ddd} that \eqref{bbb} tends to zero as $\varepsilon\to 0$. In addition, 
		\begin{align*}
			& \left|\left(\int_{\tau-\varepsilon}^{\tau}+\int_{t}^{t+\varepsilon}\right)\int_{y(t^\prime)}^{y(t^\prime)+\varepsilon}u_{x}\left(\psi_{t}^{\varepsilon}+\left(\frac{s}{\mu}u-\frac{\eta}{\mu}\right)\psi_{x}^{\varepsilon}\right)dxdt^\prime\right| \\
			\leq&\left(\int_{\tau-\varepsilon}^{\tau}+\int_{t}^{t+\varepsilon}\right)\left(\int_{y(t^\prime)}^{y(t^\prime)+\varepsilon}u_{x}^{2}dx\right)^{\frac{1}{2}}\left(\int_{y(t^\prime)}^{y(t^\prime)+\varepsilon}\left(\psi_{t}^{\varepsilon}+\left(\frac{s}{\mu}u-\frac{\eta}{\mu}\right)\psi_{x}^{\varepsilon}\right)^{2}dx\right)^{\frac{1}{2}}dt^\prime \\
			\leq&\left(\int_{\tau-\varepsilon}^{\tau}+\int_{t}^{t+\varepsilon}\right)\|u(t^\prime)\|_{H^{1}}\left(\int_{y(t^\prime)}^{y(t^\prime)+\varepsilon}C\varepsilon^{-2}\left(1+\|u\|_{L^{\infty}}^{2}\right)dx\right)^{\frac{1}{2}}dt^\prime\leq C\varepsilon^{\frac{1}{2}}\to0
		\end{align*}
		as $\varepsilon\to0$. Hence, we conclude that
		$$\lim_{\varepsilon\to0}\int_{\tau-\varepsilon}^{t+\varepsilon}\int_{y(t^\prime)}^{y(t^\prime)+\varepsilon}u_{x}\left(\psi_{t}^{\varepsilon}+\left(\frac{s}{\mu}u-\frac{\eta}{\mu}\right)\psi_{x}^{\varepsilon}\right)dxdt^\prime=0.$$
		
		\textbf{Step 6.} Using the uniqueness of $\beta(t)$, we obtain the uniqueness of $y(t)$.
	\end{proof}
	
	\begin{lemm}\label{beta_monotone}
		Let $u=u(t,x)$ be a conservative solution of \eqref{eq1} with $s\neq0$ and $t\mapsto\beta(t;\tau,\bar{\beta})$ the solution of
		\begin{equation}\label{beta(t;tau)}
			\beta(t)=\bar{\beta}+\int_\tau^tG(t^\prime,\beta(t^\prime))dt^\prime,
		\end{equation}
		where $G$ is defined by \eqref{G(t,beta)}. Then the map $\bar{\beta}\mapsto\beta(t;\tau,\bar{\beta})$ is strictly increasing and Lipschitz continuous.
		%Then there exists a constant $C$ such that, for any two initial	data $\bar{\beta}_1$, $\bar{\beta}_2$ and any $t,\tau\geq0$ the corresponding solutions satisfy
		
		%Moreover, $\beta(t;\tau,\bar{\beta})$ is strictly increasing with respect to $\bar{\beta}$.
	\end{lemm}
	\begin{proof}
		We may assume without loss of generality that $\tau<t$. It follows from \eqref{beta(t;tau)} that
		\begin{align*}
			|\beta(t;\tau,\bar{\beta}_{1})-\beta(t;\tau,\bar{\beta}_{2})|
			& \leq|\bar{\beta}_{1}-\bar{\beta}_{2}|+\int_{\tau}^{t}\left|G(t^\prime,\beta(t^\prime;\tau,\bar{\beta}_{1}))-G(t^\prime,\beta(t^\prime;\tau,\bar{\beta}_{2}))\right|dt^\prime \\
			& \leq|\bar{\beta}_{1}-\bar{\beta}_{2}|+C\int_{\tau}^{t}|\beta(t^\prime;\tau,\bar{\beta}_{1})-\beta(t^\prime;\tau,\bar{\beta}_{2})|dt^\prime.
		\end{align*}
		An application of the Gronwall lemma leads to
		$$|\beta(t;\tau,\bar{\beta}_{1})-\beta(t;\tau,\bar{\beta}_{2})|\leq e^{C|t-\tau|}|\bar{\beta}_{1}-\bar{\beta}_{2}|.$$
		In addition, for any $\bar{\beta}_{1}>\bar{\beta}_{2}$,
		\begin{align*}
			\beta(t;\tau,\bar{\beta}_{1})-\beta(t;\tau,\bar{\beta}_{2})
			& =\bar{\beta}_{1}-\bar{\beta}_{2}+\int_{\tau}^{t}G(t^\prime,\beta(t^\prime;\tau,\bar{\beta}_{1}))-G(t^\prime,\beta(t^\prime;\tau,\bar{\beta}_{2}))dt^\prime \\
			& \geq(\bar{\beta}_{1}-\bar{\beta}_{2})(1-C_{u_0}(t-\tau)),
		\end{align*}
		which gives monotonicity when $t$ is sufficiently close to $\tau$. By the continuous method, this property can be extended to all $t$ of interest.
	\end{proof}
	
	\begin{lemm}\label{lip_4}
		Let $u=u(t,x)$ be a conservative solution of \eqref{eq1} with $s\neq0$. Then the map $(t,\beta)\mapsto u(t,\beta)\triangleq u(t,y(t,\beta))$ is Lipschitz continuous with a constant depending only on $\|u_0\|_{H^1}$.
	\end{lemm}
	\begin{proof}
		The Lipschitz continuity of $\beta\mapsto u(t,\beta)$ is given by Lemma \ref{lip_3}. Thus, it remains to verify that $t\mapsto u(t,\beta)$ is Lipschitz continuous.
		
		Let $\beta(t)$ be the solution of \eqref{beta(t;tau)}. Then \eqref{lip_u_beta} and \eqref{u_Px} yield
		\begin{align*}
			&|u(t,y(t,\bar{\beta}))-u(\tau,y(\tau,\bar{\beta}))| \\
			\leq&|u(t,y(t,\bar{\beta}))-u(t,y(t,\beta(t)))|+|u(t,y(t,\beta(t)))-u(\tau,y(\tau,\beta(\tau)))| \\
			\leq&\frac{1}{2}|\beta(t)-\bar{\beta}|+|t-\tau|\|P_{x}\|_{L^{\infty}} \\
			\leq&|t-\tau|\left(\frac{1}{2}\|G\|_{L^{\infty}}+\|P_{x}\|_{L^{\infty}}\right)\\
			\leq& C_{u_0}|t-\tau|.\qedhere
		\end{align*}
	\end{proof}	
	
	\subsection{Proof of uniqueness}
	This subsection is devoted to the proof of Theorem \ref{thm_uniqueness}. We study how the gradient $u_x$ of a conservative solution varies along a good characteristic. Then we construct a semi-linear system admitting a unique solution, which in turn proves the uniqueness of the conservative solution $u$ in the original variables.
	
	\begin{proof}[Proof of Theorem \ref{thm_uniqueness}]
		The proof proceeds in five steps.
		
		\textbf{Step 1.} It follows from Lemmas \ref{lip_3} and \ref{lip_4} that $(t,\beta)\mapsto (y,u)(t,\beta)$ is Lipschitz continuous. Similarly, we can check that the maps $\beta\mapsto G(t,\beta)\triangleq G(t,y(t,\beta))$ and $\beta\mapsto P_x(t,\beta)\triangleq P_x(t,y(t,\beta))$ are also Lipschitz continuous. By Rademacher’s theorem, the partial derivatives $y_t$, $y_\beta$, $u_t$, $u_\beta$, $G_\beta$, $P_{x,\beta}$ exist almost everywhere. Furthermore, a.e. point $(t,\beta)$ is a Lebesgue point for these derivatives. Let $t\mapsto\beta(t,\bar{\beta})$ be the unique solution of \eqref{beta(t;tau)} with $\tau=0$. We see from Lemma \ref{beta_monotone} that the following assertion holds for a.e. $\bar{\beta}$.
		
		\textbf{(GC)} For a.e. $t>0$, the point $(t,\beta(t,\bar{\beta}))$ is a Lebesgue point for the partial derivatives $y_t$, $y_\beta$, $u_t$, $u_\beta$, $G_\beta$, $P_{x,\beta}$. Moreover, $y_\beta(t,\beta(t,\bar{\beta}))>0$ for a.e. $t>0$.
		
		If (GC) holds, we then call $t\mapsto y(t,\beta(t,\bar{\beta}))$ a good characteristic.
		
		\textbf{Step 2.} We aim to obtain an ODE for $y_\beta$ and $u_\beta$ along a good characteristic. Let $t\mapsto\beta(t,\bar{\beta})$ still denote the unique solution of \eqref{beta(t;tau)} with $\tau=0$. Assume that $t\mapsto y(t,\beta(t,\bar{\beta}))$ is a good characteristic. Differentiating \eqref{beta(t;tau)} with respect to $\bar{\beta}$ yields		
		\begin{align}\label{11}
			\frac{\partial}{\partial\bar{\beta}}\beta(t,\bar{\beta})=1+\int_{0}^{t}G_{\beta}(t^\prime,\beta(t^\prime,\bar{\beta}))\cdot\frac{\partial}{\partial\bar{\beta}}\beta(t^\prime,\bar{\beta})dt^\prime.
		\end{align}
		Next, integrating \eqref{partial_t y(t)} over $(0,t)$ gives 
		\begin{align*}
			y(t,\beta(t,\bar{\beta}))=y(0,\bar{\beta})+\int_0^t\frac{s}{\mu}u(t^\prime,y(t^\prime,\beta(t^\prime,\bar{\beta})))-\frac{\eta}{\mu}dt^\prime.
		\end{align*}
		Differentiating the above equation with respect to $\bar{\beta}$, we have
		\begin{align}\label{22}
			y_{\beta}(t,\beta(t,\bar{\beta}))\cdot\frac{\partial}{\partial\bar{\beta}}\beta(t,\bar{\beta})=y_{\beta}(0,\bar{\beta})+\int_{0}^{t}\frac{s}{\mu}u_{\beta}(t^\prime,\beta(t^\prime,\bar{\beta}))\cdot\frac{\partial}{\partial\bar{\beta}}\beta(t^\prime,\bar{\beta})dt^\prime.
		\end{align}
		Finally, differentiating \eqref{u_Px} with respect to $\bar{\beta}$, we get
		\begin{align}\label{33}
			u_{\beta}(t,\beta(t,\bar{\beta}))\cdot\frac{\partial}{\partial\bar{\beta}}\beta(t,\bar{\beta})=u_{\beta}(0,\bar{\beta})-\int_{0}^{t}P_{x,\beta}(t^\prime,\beta(t^\prime,\bar{\beta}))\cdot\frac{\partial}{\partial\bar{\beta}}\beta(t^\prime,\bar{\beta})dt^\prime.
		\end{align}
		In particular, the left hand sides of \eqref{11}-\eqref{33} are absolutely continuous. Combining these three equations, we obtain
		\begin{equation}
			\left\{
			\begin{array}{l}
				\frac{d}{dt}y_\beta+G_\beta y_\beta=\frac{s}{\mu}u_\beta, \\
				\frac{d}{dt}u_\beta+G_\beta u_\beta=-P_{xx}\cdot y_\beta,
			\end{array}\right.
		\end{equation}
		where
		$$
		P_{xx}=-\frac{1}{\mu}\left(-\left(\frac{A}{2}+\frac{s}{2\mu}\right)u^2+\frac{s}{2}u_x^2-\frac{B}{m+1}u^{m+1}+\left(2k+\frac{\eta}{\mu}\right)u-P\right).$$
		
		\textbf{Step 3.} We now return to the original $(t,x)$ coordinates and derive an evolution equation for $u_x$ along a good characteristic. Suppose that $\bar{x}$ is a Lebesgue point for the map $x\mapsto u_x(0,x)$. Let $\bar{\beta}$ be such that $\bar{x}=y(0,\bar{\beta})$ and assume that $t\mapsto y(t,\beta(t,\bar{\beta}))$ is a good characteristic. 
		We see from \eqref{y(t,beta)_absolutely continuous} that for $t\notin\mathcal{N}$, 
		$$y_\beta(t,\beta(t,\bar{\beta}))=\frac{1}{1+u_x^2(t,y(t,\beta(t,\bar{\beta})))}>0.$$
		Whenever $y_\beta>0$, along the characteristic through $(0,\bar{x})$ the partial derivative $u_x$ can be computed as
		$$u_x\left(t,y(t,\beta(t,\bar{\beta}))\right)=\frac{u_\beta(t,\beta(t,\bar{\beta}))}{y_\beta(t,\beta(t,\bar{\beta}))}.$$
		Thus, it follows from Step 2 that as long as $y_\beta>0$, the map $t\mapsto u_x(t,y(t,\beta(t,\bar{\beta})))$ is absolutely continuous and satisfies
		\begin{align*}
			\frac{d}{dt}u_{x}(t,y(t,\beta(t,\bar{\beta})))=&\frac{d}{dt}\left(\frac{u_{\beta}}{y_{\beta}}\right)=\frac{-y_{\beta}\left(P_{xx}y_{\beta}+G_{\beta}u_{\beta}\right)-u_{\beta}\left(\frac{s}{\mu}u_{\beta}-G_{\beta}y_{\beta}\right)}{y_{\beta}^{2}}=-P_{xx}-\frac{s}{\mu}\frac{u_{\beta}^{2}}{y_{\beta}^{2}},
		\end{align*}
		which implies
		\begin{align}
			&\frac{d}{dt}\arctan u_{x}(t,y(t,\beta(t,\bar{\beta})))=\frac{1}{1+u_{x}^{2}}\cdot\frac{d}{dt}u_{x}=\left(-P_{xx}-\frac{s}{\mu}\frac{u_{\beta}^{2}}{y_{\beta}^{2}}\right)y_{\beta}\notag \\
			=&\frac{1}{\mu}\left(-\left(\frac{A}{2}+\frac{s}{2\mu}\right)u^2-\frac{s}{2}u_x^2-\frac{B}{m+1}u^{m+1}+\left(2k+\frac{\eta}{\mu}\right)u-P\right)y_{\beta}\notag\\
			=&\frac{1}{\mu}\left(-\left(\frac{A}{2}+\frac{s}{2\mu}\right)u^2-\frac{B}{m+1}u^{m+1}+\left(2k+\frac{\eta}{\mu}\right)u-P+\frac{s}{2}\right)y_{\beta}-\frac{s}{2\mu}.
		\end{align}
		
		\textbf{Step 4.} Define 
		\begin{equation}
			v\triangleq\left\{
			\begin{array}{cl}
				2\arctan u_x & \quad\mathrm{if}\quad 0<y_\beta\leq1, \\
				\pi & \quad\mathrm{if}\quad y_\beta=0.
			\end{array}\right.
		\end{equation}
		It is easy to check that
		\begin{align*}
			y_\beta=\frac{1}{1+u_x^2}=\cos^2\frac{v}{2},\quad\frac{u_x}{1+u_x^2}=\frac{1}{2}\sin v,\quad\frac{u_x^2}{1+u_x^2}=\sin^2\frac{v}{2}.
		\end{align*}
		Based on the preceding analysis, we obtain the semi-linear system
		\begin{equation}\label{semi_unique}
			\left\{
			\begin{array}{l}
				\frac{d}{dt}\beta(t,\bar{\beta}) =G(t,\beta(t,\bar{\beta})), \\
				\frac{d}{dt}y(t,\beta(t,\bar{\beta})) =\frac{s}{\mu}u(t,\beta(t,\bar{\beta}))-\frac{\eta}{\mu}, \\
				\frac{d}{dt}u(t,\beta(t,\bar{\beta})) =-P_{x}(t,\beta(t,\bar{\beta})), \\
				\frac{d}{dt}v(t,\beta(t,\bar{\beta})) =\frac{2}{\mu}\cos^2\frac{v}{2}\left(-(\frac{A}{2}+\frac{s}{2\mu})u^2-\frac{B}{m+1}u^{m+1}+(2k+\frac{\eta}{\mu})u-P+\frac{s}{2}\right)-\frac{s}{\mu},
			\end{array}\right.
		\end{equation}
		with the initial condition
		\begin{equation}\label{semi_unique_initial}
			\left\{
			\begin{array}{l}
				\beta(0,\bar{\beta})= \bar{\beta}, \\
				y(0,\bar{\beta})= y_0(\bar{\beta}), \\
				u(0,\bar{\beta})= u_0(y(0,\bar{\beta})), \\
				v(0,\bar{\beta})= 2\arctan u_{0x}(y(0,\bar{\beta})),
			\end{array}\right.
		\end{equation}
		where $P$ and $P_x$ can be expressed in terms of the variable $\beta$ as
		\begin{align*}
			P(y(\beta)) =&\frac{1}{2\sqrt{\mu}}\int_{-\infty}^{\infty}\exp\left\{-\frac{1}{\sqrt{\mu}}\left|\int_{\beta}^{\beta^{\prime}}\cos^{2}\frac{v(s)}{2}ds\right|\right\} \\
			& \cdot\left\{\left[-(\frac{A}{2}+\frac{s}{2\mu})u^2(\beta^{\prime})-\frac{B}{m+1}u^{m+1}(\beta^{\prime})+(2k+\frac{\eta}{\mu})u(\beta^{\prime})\right]\cos^2\frac{v(\beta^{\prime})}{2}+\frac{s}{2}\sin^2\frac{v(\beta^{\prime})}{2}\right\}d\beta^{\prime},\\
			P_x(y(\beta)) =&\frac{1}{2\mu}\left(\int_\beta^\infty-\int_{-\infty}^\beta\right)\exp\left\{-\frac{1}{\sqrt{\mu}}\left|\int_{\beta}^{\beta^{\prime}}\cos^{2}\frac{v(s)}{2}ds\right|\right\} \\
			& \cdot\left\{\left[-(\frac{A}{2}+\frac{s}{2\mu})u^2(\beta^{\prime})-\frac{B}{m+1}u^{m+1}(\beta^{\prime})+(2k+\frac{\eta}{\mu})u(\beta^{\prime})\right]\cos^2\frac{v(\beta^{\prime})}{2}+\frac{s}{2}\sin^2\frac{v(\beta^{\prime})}{2}\right\}d\beta^{\prime}.
		\end{align*}
		By the Lipschitz continuity of all coefficients, the Cauchy problem \eqref{semi_unique}-\eqref{semi_unique_initial} has a unique global solution.
		
		\textbf{Step 5.} Let $u$ and $\tilde{u}$ be two conservative solutions of \eqref{eq1} with the same initial data $u_0\in H^1(\mathbb{R})$. The corresponding Lipschitz continuous maps $\beta\mapsto y(t,\beta)$, $\beta\mapsto\tilde{y}(t,\beta)$ are strictly increasing for a.e. $t\geq0$. According to Lemma \ref{beta_monotone}, the maps $\bar{\beta}\mapsto\beta(t,\bar{\beta})$ and $\bar{\beta}\mapsto\tilde{\beta}(t,\bar{\beta})$ are strictly increasing. Thus, we see that the maps $\bar{\beta}\mapsto y(t,\beta(t,\bar{\beta}))$, $\bar{\beta}\mapsto\tilde{y}(t,\tilde{\beta}(t,\bar{\beta}))$ are strictly increasing for a.e. $t\geq0$. Consequently, they have continuous inverses $x\mapsto\beta^{-1}(t,x)$, $x\mapsto\tilde{\beta}^{-1}(t,x)$.
		
		It follows from Step 4 that 
		$$y(t,\beta(t,\bar{\beta}))=\tilde{y}(t,\tilde{\beta}(t,\bar{\beta})),\quad u(t,\beta(t,\bar{\beta}))=\tilde{u}(t,\tilde{\beta}(t,\bar{\beta})),$$
		which also implies
		$$\beta^{-1}(t,x)=\tilde{\beta}^{-1}(t,x).$$
		Hence, for a.e. $t\geq0$, we have
		$$u(t,x)=u(t,\beta(t,\beta^{-1}(t,x)))=\tilde{u}(t,\tilde{\beta}(t,\tilde{\beta}^{-1}(t,x)))=\tilde{u}(t,x).$$
		This completes the proof of uniqueness.
	\end{proof}

	\smallskip
	\noindent\textbf{Acknowledgments.} This work was supported by the National Natural Science Foundation of China (No.12571261).
	
	%The authors thank the referee for valuable comments and suggestions.
	
	\noindent\textbf{Data Availability.}
	No data were used for the research described in the article.
	
	\phantomsection
	\addcontentsline{toc}{section}{\refname}
	%Ìí¼Ó²Î¿¼ÎÄÏ×µ½ÊéÇ©£¬ºê°ü hyperref
	\bibliographystyle{abbrv} %plain ,%alpha, %abbrv
	\bibliography{Feneref}
	
\end{document}